\documentclass[letterpaper,11pt]{article}
\usepackage{color}

\usepackage[small]{titlesec}
\usepackage{theorem,amsmath,amssymb,amscd,comment}
\usepackage[all]{xy}
\usepackage{booktabs}
\usepackage[hyperfootnotes=false, colorlinks, linkcolor={blue}, citecolor={magenta}, filecolor={blue}, urlcolor={blue}]{hyperref}
\theoremstyle{change}
\allowdisplaybreaks
\newcommand{\A}{{\mathbb A}}

\newcommand{\Q}{{\mathbb Q}}
\newcommand{\Z}{{\mathbb Z}}
\newcommand{\R}{{\mathbb R}}
\newcommand{\C}{{\mathbb C}}

\newcommand{\p}{\mathfrak p}
\newcommand{\OF}{{\mathfrak o}}
\newcommand{\GL}{{\rm GL}}

\newcommand{\SL}{{\rm SL}}
\newcommand{\SO}{{\rm SO}}

\newcommand{\GSp}{{\rm GSp}}

\newcommand{\sgn}{{\rm sgn}}
\newcommand{\St}{{\rm St}}

\newcommand{\I}{{\rm I}}

\renewcommand{\P}{\mathfrak{P}}

\newcommand{\aaa}{{\bf a}}
\newcommand{\bb}{{\bf b}}
\newcommand{\cc}{{\bf c}}
\newcommand{\dd}{{\bf d}}

\newcommand{\mat}[4]{{\setlength{\arraycolsep}{0.5mm}\left[
\begin{array}{cc}#1&#2\\#3&#4\end{array}\right]}}
\newcommand{\qed}{\hspace*{\fill}\rule{1ex}{1ex}}
\newcommand{\forget}[1]{}

\def\qdots{\mathinner{\mkern1mu\raise0pt\vbox{\kern7pt\hbox{.}}\mkern2mu
\raise3.4pt\hbox{.}\mkern2mu\raise7pt\hbox{.}\mkern1mu}}

\newenvironment{proof}{\vspace{1ex}\noindent{\it Proof.}\hspace{0.1em}}
	{\hfill\qed\vspace{2ex}}

\newtheorem{lemma}{Lemma.}[section]
\newtheorem{theorem}[lemma]{Theorem.}
\newtheorem{corollary}[lemma]{Corollary.}
\newtheorem{proposition}[lemma]{Proposition.}

\newtheorem{remark}[lemma]{Remark.}

\begin{document}
\title{Restrictions of Eisenstein series and Rankin-Selberg convolution}
\author{Rodney Keaton,  Ameya Pitale}
\date{}
\maketitle
\tableofcontents

\section{Introduction}
\subsection{Triple product $L$-functions}
Central values of $L$-functions play an important role in number theory. If $\pi_1, \pi_2$ and $\pi_3$ are three cuspidal automorphic representations of $\GL(2,\A_F)$, for a number field $F$, then one can consider the central value $L(1/2, \pi_1 \times \pi_2 \times \pi_3)$ of the degree $8$ triple product $L$-function. It was conjectured by Jacquet that this central value is non-zero if and only if there is a quaternion algebra $D$ over $F$ such that 
$$\int\limits_{\A_F^\times D^\times(F) \backslash D^\times(\A_F)} f_1^D(x) f_2^D(x) f_3^D(x) \, d^\times x \neq 0.$$
Here, $f_i^D$ are cusp forms in $\pi_i^D$, the cuspidal automorphic representation of $D^\times(\A)$ obtained by Jacquet-Langlands correspondence from $\pi_i$. This was proven by Harris and Kudla \cite{HK}. One can look at the more general situation of a cubic extension $E$ of $F$ and consider an irreducible cuspidal automorphic representation $\Pi$ of $\GL(2,\A_E)$. In this context, the analogue of Jacquet's conjecture has been proven by Prasad and Schulze-Pillot in \cite{PSP}. Explicit formulas relating the $L$-value and the integral have been obtained by several authors (see \cite{BSP}, \cite{GK}, \cite{I}, \cite{NPS} and \cite{W}).

In this paper, we look at the special case of $E = L \times F$, where $L$ is a quadratic extension of $F$. We consider a cuspidal automorphic representation $\pi$ on $\GL(2, \A_F)$ and an induced representation ${\rm Ind_B^{\GL(2)}}(\Omega_1, \Omega_2)$ on $\GL(2, \A_L)$. Here, $\Omega_1, \Omega_2$ are characters in $L^\times \backslash \A_L^\times$ and $B$ is the Borel subgroup of $\GL(2)$. Given a smooth section $f$ in the induced representation, let $E(g,s; f)$  be the standard Eisenstein series (see (\ref{Eis-ser-defn})). For $\phi \in \pi$, we consider the pullback of the Eisenstein series given by the integral 
\begin{equation}\label{Int-defn-intro}
Z(s, f, \phi) = \int\limits_{\A_F^\times \GL(2,F) \backslash \GL(2, \A_F)} E(h,s;f) \phi(h) dh.
\end{equation}
We compute the above integral under certain  assumptions on the ramification of the local representations and characters. We assume that $\pi_\infty$ is a holomorphic discrete series. For an unramified place $v$, we assume that $\pi_v$ is either unramified or an unramified twist of the Steinberg representation. Let $\Omega_1, \Omega_2$ be such that $c(\Omega_{2,v}) = 0$ and  $c(\Omega_{1,v}) = 0$ or $1$. In the latter case,  $\pi_v$ is ramified. Also, assume that the archimedean components of the characters match with the archimedean component of $\pi$ (See Section \ref{global-section} for details). It turns out that $Z(s,f,\phi) = 0$ if the central $L$-value $L(1/2, {\rm BC}(\pi) \times \Omega) =0$. Here, ${\rm BC}(\pi)$ is the base change of $\pi$ to $\GL(2, \A_L)$ and $\Omega$ is the character on $\A_L^\times$ defined by $\Omega(z) := \Omega_1^{-1}(\bar{z}) \Omega_2^{-1}(z)$. This vanishing condition follows from the criteria for existence of Waldspurger models and we will explain it in the next section.  Assuming $L(1/2, {\rm BC}(\pi) \times \Omega)  \neq 0$, we can choose $f$ and $\phi$ (see Section \ref{global-section}) so that
\begin{equation}\label{Global-formula-intro}
Z(s, f, \bar\phi) = \frac{L(2s+\frac 12, \tilde\pi \times \Omega_1|_{\A^\times})}{L(2s+1, \Omega_1 \Omega_2^{-1})} \prod\limits_{p \leq \infty} Y_p(s).
\end{equation}
Here, $\tilde\pi$ is the contragredient representation of $\pi$. The values of $Y_p(s)$ are explicitly computed and, for almost all finite $p$, the term $Y_p(s) = 1$. The exact value of $Y_p(s)$ is given in Theorem  \ref{global-int-thm}.

\subsection{Waldspurger models}
Unwinding the integral (\ref{Int-defn-intro}), we can deduce that $Z(s, f, \phi)$ is Eulerian from the following formula.
$$Z(s, f, \phi) = \int\limits_{T(\A_F)\backslash \GL(2, \A_F)} f(\eta h,s) B_\phi(h) dh.$$
Here, $\eta$ is the non-trivial representative of $B(L) \backslash \GL(2,L) / \GL(2,F)$. The torus $T(F)$ is the subgroup of $\GL(2, F)$ isomorphic to $L^\times$.  $B_\phi$ is the period defined by
$$B_\phi(g) := \int\limits_{Z(\A_F)T(F) \backslash T(\A_F)} \phi(tg)\Omega^{-1}(t)dt.$$
It has been shown in \cite{Wald85} that a necessary condition for $B_\phi$ to be non-zero is that $L(1/2, {\rm BC}(\pi) \times \Omega) \neq 0$. We assume this non-vanishing condition. The map $\phi \mapsto B_\phi$ gives a global $\Omega$-Waldspurger model for $\pi$. The Waldspurger model gives a realization of the representation $\pi$ in terms of $\C$-valued functions on the group, which transform by the character $\Omega$ upon left translation by the torus $T$. 

We can now choose a factorizable section $f(h, s) = \prod_v f_v(h_v, s)$, and the uniqueness of the Waldspurger models allows us to write $B_\phi(h) = \prod_v B_v(h_v)$. This gives us $Z(s, f, \phi) = \prod_v Z_v(s)$, where 
\begin{equation}\label{local-int-intro}
Z_v(s) = \int\limits_{T(F_v) \backslash \GL(2, F_v)} f_v(\eta h_v, s) B_v(h_v) dh_v.
\end{equation}

\subsection{Explicit formulas for new-forms in local Waldspurger models}
The key to computing (\ref{local-int-intro}) is choosing appropriate local vectors $f_v$ and $B_v$. There are two reasonable choices for $B_v$ in the non-archimedean case -- the new-form or the Gross-Prasad test vector. For the application that we have in mind towards the conjectures of Tonghai Yang, we will choose the new-form. For more on the Gross-Prasad test vectors see \cite{FMP}. In the archimedean case, we will assume that $\pi_\infty$ is a holomorphic discrete series with lowest non-negative weight $\ell$. The vector $B_\infty$ will be chosen to be the weight $\ell$ vector. The choice for $f_v$ is more straightforward. We choose the vector in the local induced representation that is right invariant under an appropriate compact subgroup so that the integral $Z_v$ is not trivially zero. 

In order to actually compute $Z_v$ we need explicit formulas for the local vectors $B_v$. One of the main contributions of this paper is explicit formulas for certain distinguished vectors in the Waldspurger models for local representations of $\GL(2)$. 

{\it Unramified non-archimedean case}: When $\pi_v$ is unramified, we obtain explicit formulas for the spherical vector $B_0$ in the Waldspurger model. The vector is determined by its values on $\{ \mat{\varpi^m}{}{}{1} : m \geq 0\}$. We use the fact that the spherical vector is an eigenfunction of the local Hecke algebra to get recurrence relations on the above values. This allows us to obtain 
$$\sum\limits_{m \geq c(\Omega)} B_0(\mat{\varpi^m}{}{}{1}) x^m = \frac{(q-\kappa x)x^{c(\Omega)}}{\omega_\pi(\varpi)x^2- \lambda x + q}B_0(\mat{\varpi^{c(\Omega)}}{}{}{1}).$$
Here, $\lambda$ is the eigenvalue of $B_0$. Also, $\kappa$ is an explicit constant depending on the conductor $c(\Omega)$ and the ramification of $L_v/F_v$. See Proposition \ref{main-result-prop} for details. Note that the unramified computations put no restriction on the character $\Omega$ or the field extension $L$. This extends results of \cite{BFF}. 

{\it Ramified non-archimedean case}: We obtain the explicit formulas for the new-form in the Waldspurger model for the twist of the Steinberg representation of $\GL(2)$ by an unramified character $\chi_v$. When $L_v/F_v$ is a field extension, this was done in \cite{FMP}. We compute the remaining case when $L_v = F_v \oplus F_v$. Note that a necessary and sufficient condition for a local Waldspurger model to exist is that $\Omega_v \neq \chi_v \circ N_{L_v/F_v}$. We use the fact that the new-form is right invariant under the Iwahori subgroup and is an eigenfunction of the Atkin-Lehner operator and the Hecke operator.

{\it Archimedean case}: We assume that $\pi_\infty$ is the holomorphic discrete series of $\GL(2,\R)$ with lowest non-negative weight $\ell$. We compute the explicit formulas for the weight $\ell$ vector $B_0$ in the Waldspurger model for $\pi_\infty$. The key property of $B_0$ is that it is annihilated by the lowering operator in the complexified Lie algebra of $\SL(2, \R)$. We consider the action of the lowering operator on vectors in $\pi_\infty$. The criteria that $B_0$ is annihilated by the lowering operator reduces to a first order linear ordinary differential equation satisfied by $B_0$. This leads to the explicit formulas in both the cases when $L_\infty$ is split or non-split over $F_\infty = \R$. In the split case, we use these formulas to compute the local archimedean integral $Z_\infty(s)$ as follows
\begin{equation}\label{Arch-int-intro}
Z_\infty(s) = \begin{cases} iD^{-1/2} \pi & \text{ if } \ell = 2, s= 0; \\ 
2^{2-2s-\ell_2} D^{-\frac{\ell}4-s} i^{\frac{\ell}2} \pi \frac{\Gamma(2s+\frac{\ell}2-1)}{\Gamma(s) \Gamma(\frac{\ell}2+s)} & \text{ if } {\rm Re}(2s+\frac{\ell}2) > 1.\end{cases}
\end{equation}
Here, $\ell_2$ depends on $\Omega_\infty$ and $D$ is the fundamental discriminant for $L/F$. We do the split computation here because we want to apply this to the case of Hilbert modular forms. If one considers Bianchi modular forms, then one can use the explicit formulas for the weight vector in the non-split case to compute the integral as well. In that case, one has to deal with the further complication that the maximal compact is not abelian and hence, we have $K$-types of higher dimensions. We have not done that case here.

The local computations in the archimedean and non-archimedean case lead to the calculation of the local integrals $Z_v(s)$ leading to the formula (\ref{Global-formula-intro}).

\subsection{Application to Tonghai Yang's conjectures}
In \cite{Yang}, Yang constructs a Hilbert Eisenstein series  $\mathcal E((\tau_1, \tau_2), s, f)$ over a totally real extension $L/\Q$ associated to an imaginary quadratic extension $K/L$. This involves choosing the characters $\Omega_1 = \chi_{K/L}$, the character corresponding to the extension $K/L$ by class field theory, $\Omega_2 = 1$, and choosing a square-free ideal $\mathcal N$. As a function of $(\tau_1, \tau_2) \in \mathcal H^2$, Yang shows that 
$\mathcal E((\tau_1, \tau_2), s, f)$ is a Hilbert Eisenstein series  of weight $(1,1)$, of square-free level $\mathcal N d_{K/L}$ and Nebentypus character corresponding to $\chi_{K/L}$. He gives explicit formulas for the Fourier coefficients of $\mathcal E((\tau_1, \tau_2), s, f)$, a criteria for non-vanishing and shows that it is holomorphic for $s=0$. By restriction to the diagonal, we get that $\mathcal E((\tau, \tau), 0, f)$ is a holomorphic modular form of weight $2$, square-free level $N$ and Nebentypus character $\psi$. Here, $N$ and $\psi$ depend on $K, L$ and $\mathcal N$. 

By allowing $K$ and $L$ to vary subject to certain conditions, Yang obtains a family of such Hilbert Eisenstein series. He conjectures that the restriction of these Eisenstein series to the diagonal forms a spanning set for the space of holomorphic modular forms of weight $2$, level $N$ and Nebentypus character $\psi$. One of the key steps towards this conjecture is the following --

\emph{Given a cusp form $\Phi \in S_2(N, \psi)$, does there exist a choice of $K, L, \mathcal N$, such that the corresponding Hilbert Eisenstein series $\mathcal E((\cdot, \cdot), s, f)$ satisfies}
$$\langle \mathcal E((\cdot, \cdot), s, f)|_{\Delta \mathcal H}, \Phi \rangle \neq 0.$$
Here, $\langle \,,\, \rangle$ is the Petersson inner product. Suppose $\Phi$ is a Hecke eigenform, then let $\phi$ be the function on $\GL(2, \A)$ corresponding to $\Phi$ and let $\pi$ be the irreducible, cuspidal automorphic representation of $\GL(2, \A)$ corresponding to $\Phi$. Let $E(g, s, f)$ be the Eisenstein series on $\GL(2, \A_L)$ corresponding to $\mathcal E$. Then, we show in Proposition \ref{Petersson-norm-prop} that 
\begin{equation}\label{innerprod-reln-intro}
Z(s, f, \bar\phi) = {\rm vol}(\Gamma_0(N) \backslash \mathcal H) \langle \mathcal E((\cdot, \cdot), s, f)|_{\Delta \mathcal H}, \Phi \rangle.
\end{equation}
Using (\ref{Global-formula-intro}), we get, in Corollary \ref{classical-cor}
$$\langle \mathcal E((\cdot, \cdot), 0, f)|_{\Delta \mathcal H}, \Phi \rangle \neq 0 \text{ if and only if } L(1/2, \pi) \neq 0 \text{ and } L(1/2, {\rm BC}(\pi) \times \chi_{K/L}) \neq 0.$$
If $L(1/2, \pi) = 0$, we immediately get that $\Phi$ cannot be in the span of the Hilbert Eisenstein series. In case $L(1/2, \pi) \neq 0$, then using the results of Friedberg and Hoffstein in \cite{FH}, one can obtain characters $\chi_{K/L}$ such that $L(1/2, {\rm BC}(\pi) \times \chi_{K/L}) \neq 0$. If we expand the family of Hilbert Eisenstein series by allowing more general choices of $\Omega_1, \Omega_2$, then the criteria of non-vanishing of the inner product changes from $L(1/2, \pi) \neq 0$ to a twist $L(1/2, \pi \times \chi) \neq 0$, for a suitable character $\chi$. Again using \cite{FH}, there is now a chance to achieve this. This is the advantage of computing the global integral $Z(s, f, \phi)$ for as general a choice of data as possible. 

The formula (\ref{innerprod-reln-intro}) relating the inner product to the global integral is the reason for choosing new-forms for local vectors in the Waldspurger models of the local representations. Also, we have not considered highly ramified local representations $\pi_v$ because they do not appear in considerations for the application to Tonghai Yang's conjecture. 

Let us also remark that it is not possible to extend the ideas of Tonghai Yang in a naive manner to obtain spanning sets for modular forms of weight $\ell > 2$. This is because (\ref{Arch-int-intro}) easily gives us $Z_\infty(0) = 0$ for $s = 0$ and $\ell > 2$. 

\subsection{Previous work}
Observe that the computations mentioned above work only when the holomorphic cusp form $\Phi$ is a Hecke eigenform. Even if we get non-vanishing of Petersson inner product for all Hecke eigenforms, it does not imply non-vanishing for non-Hecke eigenforms. In a certain special case, Yingkun Li in \cite{Li} has obtained a complete answer. Fix an odd, square-free integer $N$. Consider any two coprime, negative, fundamental discriminants $d_1, d_2$ such that $\Big(\frac{d_1}p\Big) = \Big(\frac{d_2}p\Big) = -1$, for all $p | N$. Let $K = \Q(\sqrt{d_1}, \sqrt{d_2})$ and $L = \Q(\sqrt{d_1d_2})$ and  let $\mathcal N$ be a square-free ideal in $L$ with an odd number of prime divisors such that $\mathcal N \cap \Z = N \Z$. Li proves that the span of the restriction of the Hilbert Eisenstein series corresponding to $d_1, d_2, \mathcal N$, varying under the above restrictions, is precisely the space spanned by the Eisenstein series $E_{2,N} \in M_2(N)$ and all cuspidal eigenforms $\Phi \in M_2(N)$ satisfying $L(1/2, \Phi) \neq 0$. The key to obtaining this result is once again the computation of the Petersson inner product. In this particular case, Li makes use of the explicit Fourier coefficients of the Hilbert Eisenstein series to show that the restriction is a Shimura lift of a weight $3/2$ modular form. This leads to an explicit formula for the inner product in terms of the central value of the $L$-function and certain Fourier coefficients of the half integral weight modular form. 

These classical methods cannot be easily generalized to other choices of $K, L, \mathcal N$ from Yang's conjectures. In fact, a question from Li to compute the inner product in an adelic setting was the starting point of this current paper. 

It should be remarked that this inner product has been considered in \cite{Hu}. In \cite{Hu} the author also considers certain non-squarefree level cases by choosing the Gross-Prasad test vector. Let us point out that the computation technique in \cite{Hu} is completely different to that used in the present paper. 

\subsection{Structure of the paper}

In Section \ref{Prelim}, we introduce our basic objects of study as well as state and prove that the global integral is Eulerian. In Section \ref{ExplicitValues} we present the calculation of explicit values of a new vector in the Waldspurger model of an unramified principal series (Section \ref{ExplicitValuesUnram}) and in the Waldspurger model of an unramified twist of a Steinberg representation (Section \ref{ExplicitValuesSteinberg}). In Section \ref{IntegralCalculation} we perform the local integral calculations needed for our inner product. In Section \ref{ExplicitValuesArch} we present the analogous explicit values of the Waldspurger model at the archimedean places, and also compute the local inner product in the split case, i.e., when the quadratic field extension is totally real. Finally, in Section \ref{GlobalIntegral}, we combine our local calculations to obtain the calculation of the global integral. Also, in this section we relate our integral calculation to the inner product mentioned above and give a case which is relevant to the conjectures of Yang.

\section{Preliminaries}\label{Prelim}
\subsection{Eisenstein series and Waldspurger models}
Let $F$ be a number field. Let $\aaa,\bb,\cc \in F$ such that $\dd:=\bb^2-4\aaa\cc \neq 0$. Let $L = F(\sqrt{\dd})$ be a subfield of $\C$. Let $\A$ be the ring of adeles of $F$ and $\A_L$ be the ring of adeles of $L$. Let $H$ be defined by $H(R) = \GL_2(R)$ for a ring $R$. Let $B$ be the standard Borel subgroup of $H$. Let $\pi$ be an irreducible cuspidal automorphic representation of $H(\A)$ with central character $\omega_\pi$. Let $\Omega_1, \Omega_2$ be characters of $\A_L^\times/L^\times$ such that $\Omega_1 \Omega_2 |_{\A^\times} = \omega_\pi^{-1}$. For $s \in \C$, let $I(\Omega_1, \Omega_2, s) = {\rm Ind}_{B(\A_L)}^{H(\A_L)}(\Omega_1, \Omega_2, \delta_B^s)$. Here, $\delta_B$ is the modulus character $\delta_B(\mat{u}{v}{}{w}) = |u/w|_{\A_L}$. Hence, for $f \in I(\Omega_1, \Omega_2, s)$, we have
\begin{equation}\label{ind-rep-defn}
f(\mat{u}{v}{}{w} g, s) = \Omega_1(u) \Omega_2(w) |u/w|_{\A_L}^{s+1/2}f(g,s).
\end{equation}
For any section $f \in I(\Omega_1, \Omega_2, s)$, define the Eisenstein series
\begin{equation}\label{Eis-ser-defn}
E(g,s) = E(g,s; f) = \sum\limits_{\gamma \in B(L)\backslash H(L)} f(\gamma g, s).
\end{equation}
This series is absolutely convergent for ${\rm Re}(s) > 1/2$ and has a meromorphic continuation to all of $\C$ (see \cite{L-Eisenstein}). 

For $\aaa,\bb,\cc$ as above, set
$$S = \mat{\aaa}{\bb/2}{\bb/2}{\cc}, \text{ and } \xi=\mat{\frac{\mathbf{b}}2}{\mathbf{c}}{-\mathbf{a}}{\frac{-\mathbf{b}}2}.
$$
Let $F(\xi) = \{x I_2+y \xi : x, y \in F\} \subset M_2(F)$. We have the isomorphism
$$F(\xi) \ni x I_2+y\xi\mapsto
x+y\frac{\sqrt{\mathbf{d}}}2 \in L.$$
Let 
$$T(F) = \{g \in H(F) : {}^{t}gSg = \det(g)S\}.$$
Then $T(F) = F(\xi)^\times$ and hence, $T(F) \simeq L^\times$.  Note that $T(F)$
consists of all matrices
\begin{equation}\label{TFgeq1}
g = t(x,y)=\mat{x+y\frac{\mathbf{b}}2}{\mathbf{c}y}{-\mathbf{a}y}{x-y\frac{\mathbf
{b}}2},\qquad x, y \in F, \,
\det(g)=x^2-\frac14y^2(\mathbf{b}^2-4\mathbf{a}\mathbf{c})\neq0.
\end{equation}
Let $\Omega$ be a character of $T(\A)/T(F) \simeq \A_L^\times/L^\times$ defined by
\begin{equation}\label{Omega-defn}
\Omega(z) := \Omega_1^{-1}(\bar{z}) \Omega_2^{-1}(z), \qquad \text{ for all } z \in \A_L^\times.
\end{equation}
Hence,  $\Omega |_{\A^\times} = \omega_\pi$. For $\phi \in V_\pi$, define
\begin{equation}\label{global-wald-defn}
B_\phi(g) = \int\limits_{Z_H(\A)T(F) \backslash T(\A)} \phi(tg)\Omega^{-1}(t)dt.
\end{equation}
The $\C$-vector space spanned by $\{B_\phi : \phi \in V_\pi\}$ is called the global Waldspurger model of $\pi$ of type $(S, \Omega)$. The uniqueness and criteria for existence for having such a Waldspurger model is known by \cite{Sai}, \cite{Tu}, and \cite{Wald85}. We will assume that such a Waldspurger model exists.

Let $\phi \in V_\pi$. We wish to study the integral
\begin{equation}\label{global-integral-defn}
Z(s) = Z(s, f, \phi) = \int\limits_{H(F)Z_H(\A)\backslash H(\A)} E(h,s;f) \phi(h) dh.
\end{equation}

\subsection{Basic Identity}
The first step is to show that the above integral is Eulerian. Using the Bruhat decomposition of $\GL(2)$, we get the following lemma.
\begin{lemma}\label{double-coset-decomp}
The representatives for the double cosets $B(L)\backslash H(L)/H(F)$ are given by $ I_2$ and $\eta=\mat{1}{}{\beta}{1}$, with $\beta = (\bb+\sqrt{\dd})/(2\cc)$.
\end{lemma}
Let us denote by $\Delta(F) = B(L) \cap H(F)$ and $\Delta_0(F) = \eta^{-1}B(L)\eta \cap H(F)$, subgroups of $H(F)$. 
\begin{lemma}\label{Delta0=T}
We have
$$\Delta_0(F) = T(F).$$
\end{lemma}
\begin{proof}
Let $e_1={}^{t}[1,0]$ and $e_2={}^{t}[0,1]$. Let $h \in \Delta_0$. Hence, $h\eta^{-1}e_1 = \gamma \eta^{-1}e_1$ for some $\gamma \in L^\times$, since $B(F)$ fixes the line generated by $e_1$. Let $\gamma = x+y\sqrt{\dd}$ for $x, y \in F$. We have $\eta^{-1} e_1 = e_1 - \beta e_2$. Hence $he_1 - \beta he_2 = (x+y\sqrt{\dd}) e_1 - (x+y\sqrt{\dd}) \beta e_2$. Since $h \in H(F)$, we get two equations by comparing the coefficient of $\sqrt{d}$ and the coefficient of $1$ on both sides. This gives us $he_1 = (x-\bb y) e_1 + 2 \aaa y e_2$ and $he_2 = -2\cc y e_1 + (x+\bb y)e_2$. Hence, $h = x-2y\xi \in T(F)$. The reverse implication can also be worked out similarly.
\end{proof}

By Lemmas \ref{double-coset-decomp} and \ref{Delta0=T}, we have
$$E(g,s;f) = \sum\limits_{\gamma \in B(L)\backslash H(L)} f(\gamma g, s) = \sum\limits_{\gamma \in \Delta(F)\backslash H(F)} f(\gamma g, s) + \sum\limits_{\gamma \in T(F)\backslash H(F)} f(\eta \gamma g, s).$$
Hence
$$Z(s) = \int\limits_{\Delta(F) Z_H(\A) \backslash H(\A)} f(h) \phi(h) dh + \int\limits_{T(F) Z_H(\A) \backslash H(\A)} f(\eta h) \phi(h) dh.$$
Using cuspidality of $\pi$, we have
$$\int\limits_{\Delta(F) Z_H(\A) \backslash H(\A)} f(h) \phi(h) dh = 0.$$
Thus,
$$Z(s)=\int\limits_{T(F) Z_H(\A) \backslash H(\A)} f(\eta h) \phi(h) dh,$$
which will be needed in the proof of the following proposition.
\begin{proposition}\label{Basic-identity-prop}
Let $\pi$ be an irreducible cuspidal automorphic representation of $\GL_2(\A)$ with central character $\omega_\pi$. Let $\Omega_1, \Omega_2$ be characters of $\A_L^\times/L^\times$ such $\Omega_1 \Omega_2 |_{\A^\times} = \omega_\pi$. Let $f \in I(\Omega_1, \Omega_2, s)$ and $\phi \in V_\pi$. Then we have
$$Z(s, f ,\phi) = \int\limits_{H(F)Z_H(\A)\backslash H(\A)} E(h,s;f) \phi(h) dh =  \int\limits_{T(\A)\backslash H(\A)} f(\eta h,s) B_\phi(h) dh.$$
Here, $B_\phi$ is as defined in (\ref{global-wald-defn}) with $\Omega$ defined in (\ref{Omega-defn}). Also, $\eta = \mat{1}{}{\beta}{1}$ with $\beta = (\bb+\sqrt{\dd})/(2 \cc)$. 
\end{proposition}
\begin{proof}
We have
\begin{align*}
Z(s) &= \int\limits_{T(F) Z_H(\A) \backslash H(\A)} f(\eta h) \phi(h) dh \\
&= \int\limits_{T(\A) \backslash H(\A)} \int\limits_{T(F) Z_H(\A) \backslash T(\A)} f(\eta t h) \phi(th) dt dh.
\end{align*}
For $t = x I_2+y\xi \in T(\A)$, 
we get
$$f(\eta t h, s) = f(\eta t \eta^{-1} \eta h, s) = \Omega^{-1}(x+y\sqrt{d}/2) f(\eta h,s).$$
Hence,
$$Z(s) = \int\limits_{T(\A) \backslash H(\A)} f(\eta h,s) \Big(\int\limits_{T(F) Z_H(\A) \backslash T(\A)}  \Omega^{-1}(t) \phi(th) dt \Big) dh = \int\limits_{T(\A)\backslash H(\A)} f(\eta h,s) B_\phi(h) dh,$$
as required.
\end{proof}

By the uniqueness of the Waldspurger model, we have
$$B_\phi(h) = \prod_v B_v(h_v), \qquad f(h, s) = \prod_v f_v(h_v, s)$$
where $h = \otimes' h_v$. Hence, $Z(s) = \prod_v Z_v(s)$, where
\begin{equation}\label{local-integral}
Z_v(s) = \int\limits_{T(F_v) \backslash H(F_v)} f_v(\eta_v h_v, s) B_v(h_v) dh_v.
\end{equation}

\section{Values of the newform in the Waldspurger model}\label{ExplicitValues}
In this section, we will compute the explicit values of the new vector in the Waldspurger model when the $\GL(2)$ representation is either unramified or an unramified twist of a Steinberg representation. Note, in the latter, we will recall the values computed in \cite{FMP} as well as a new calculation when the local extension $L/F$ is split.

\subsection{Set-up}
Let $F$ be a local non-archimedean field of characteristic zero. We will drop the subscript $v$ in this section. Let $\OF$ be the ring of integers of $F$, $\p$ the unique maximal ideal, $\varpi$ a uniformizer and let $q$ be the residue characteristic. Let $K = H(\OF)$ be the maximal compact subgroup of $H(F)$. 

We have fixed three elements $\mathbf{a},\mathbf{b},\mathbf{c}\in F$
such that $\mathbf{d}=\mathbf{b}^2-4\mathbf{a}\mathbf{c}\neq0$. We have
$L=F(\sqrt{\mathbf{d}})$ if $\mathbf{d}\notin F^{\times2}$, and $L=F\oplus F$
otherwise. In the latter case we consider $F$ diagonally embedded. Let
$z\mapsto\bar z$ be the obvious involution on $L$ whose fixed point set is $F$.
We define the Legendre symbol as
\begin{equation}\label{legendresymboldefeq}
 \Big(\frac L\p\Big)=\begin{cases}
                      -1&\text{if $L/F$ is an unramified field extension},\\
                      0&\text{if $L/F$ is a ramified field extension},\\
                      1&\text{if }L=F\oplus F.
                     \end{cases}
\end{equation}
We will make the following assumptions:
\begin{itemize}
 \item $\mathbf{a},\mathbf{b}\in\OF$ and $\mathbf{c}\in\OF^\times$.
 \item If $\mathbf{d}\notin F^{\times2}$, then $\mathbf{d}$ is a generator of
the discriminant of $L/F$.
 \item If $\mathbf{d}\in F^{\times2}$, then $\mathbf{d}\in\OF^\times$.
\end{itemize}
We define elements $\beta$ and $\xi_0$ of $L$ by
\begin{equation}\label{alphadefeq}
 \beta=\left\{\begin{array}{l@{\qquad\text{if }L}l}
 \displaystyle\frac{\mathbf{b}+\sqrt{\mathbf{d}}}{2\mathbf{c}}&\text{ is a
field},\\[2ex]
 \displaystyle\Big(\frac{\mathbf{b}+\sqrt{\mathbf{d}}}{2\mathbf{c}},\frac{
\mathbf{b}-\sqrt{\mathbf{d}}}{2\mathbf{c}}\Big)&=F\oplus F.
 \end{array}\right.
\end{equation}
\begin{equation}\label{xi0defeq}
 \xi_0=\left\{\begin{array}{l@{\qquad\text{if }L}l}
 \displaystyle\frac{-\mathbf{b}+\sqrt{\mathbf{d}}}{2}&\text{ is a field},\\[2ex]
 \displaystyle\Big(\frac{-\mathbf{b}+\sqrt{\mathbf{d}}}{2},\frac{-\mathbf{b}
-\sqrt{\mathbf{d}}}{2}\Big)&=F\oplus F.
 \end{array}\right.
\end{equation}
If $L$ is a field, let $\OF_L$ be its ring of integers, $\varpi_L$ a uniformizer, and $v_L$
the normalized valuation.  If $L = F \oplus F$, put $\OF_L = \OF \oplus \OF$ and $\varpi_L = (\varpi,1)$.
By Lemma 3.1.1 of \cite{PS1}, in either case,
\begin{equation}\label{integralbasiseq}
 \OF_L = \OF +\OF\beta=\OF+\OF\xi_0.
\end{equation}
Fix the ideal in $\OF_L$ given by
\begin{equation}\label{ideal defn}\renewcommand{\arraystretch}{1.3}
 \P_L := \p\OF_L = \left\{
                  \begin{array}{l@{\qquad\text{if }}l}
                    \p_L & \big(\frac L{\p}\big) = -1,\\
                    \p_L^2 & \big(\frac L{\p}\big) = 0,\\
                    \p \oplus \p & \big(\frac L{\p}\big) = 1.
                  \end{array}
                \right.
\end{equation}
Here $\p_L$ is the maximal ideal of $\OF_L$ when $L$ is a field. 
We have $\P_L^n\cap\OF=\p^n$ for all $n\geq0$.

Let us recall the embedding of $L^\times$ as a torus in $H(F)$ for convenience of
calculations. 
With $\mathbf{a},\mathbf{b},\mathbf{c}$ as above, let
$$
 S=\mat{\mathbf{a}}{\frac{\mathbf{b}}2}{\frac{\mathbf{b}}2}{\mathbf{c}},\qquad
 \xi=\mat{\frac{\mathbf{b}}2}{\mathbf{c}}{-\mathbf{a}}{\frac{-\mathbf{b}}2}.
$$
Then $F(\xi)=F \cdot I_2 +F \cdot \xi$ is a two-dimensional $F$-algebra isomorphic to
$L$. If $L$ is a field, then an isomorphism is given by $x+y\xi\mapsto
x+y\frac{\sqrt{\mathbf{d}}}2$. If $L=F\oplus F$, then an isomorphism is given by
$x+y\xi\mapsto(x+y\frac{\sqrt{\mathbf{d}}}2,x-y\frac{\sqrt{\mathbf{d}}}2)$. The
determinant map on $F(\xi)$ corresponds to the norm map on $L$. Let
\begin{equation}\label{TFdefeq}
 T(F)=\{g\in H(F):\:^tgSg=\det(g)S\}.
\end{equation}
One can check that $T(F)=F(\xi)^\times$. Note that $T(F)\cong L^\times$ via the
isomorphism $F(\xi)\cong L$. Under the same isomorphism the group
$T(\OF):=T(F)\cap K$ is isomorphic to $\OF_L^\times$. Note that $T(F)$
consists of all matrices
\begin{equation}\label{TFgeq}
g=t(x,y)=\mat{x+y\frac{\mathbf{b}}2}{\mathbf{c}y}{-\mathbf{a}y}{x-y\frac{\mathbf
{b}}2},\qquad x, y \in F, \,
\det(g)=x^2-\frac14y^2(\mathbf{b}^2-4\mathbf{a}\mathbf{c})\neq0.
\end{equation}
Let $\Omega$ be any character of $L^\times$, which we may view as a character of the torus $T(F)$.  Define
\begin{equation}\label{conductor-of-Omega}
c(\Omega) := \text{ min }\{m \geq 0 : \Omega |_{(1+\P_L^m) \cap \OF_L^\times}
\equiv 1\}.
\end{equation}
Note that this is the conductor of $\Omega$ only in the case $L/F$ is an unramified field extension. Let $\mathcal{B}(\Omega)$ be the space of all locally constant functions $B :
H(F) \rightarrow \C$ satisfying 
\begin{equation}\label{Wald-trans-prop}
B(tg) = \Omega(t) B(g) \qquad \text{ for all } t \in T(F), g \in H(F).
\end{equation}
Let $(\pi, V)$ be any infinite dimensional, irreducible, admissible representation of $H(F)$. We
say that $\pi$ has an {\it $(S, \Omega)$-Waldspurger model} if $\pi$ is isomorphic to
a subrepresentation of $\mathcal{B}(\Omega)$. We call a linear functional $\ell$
on $\pi$ an {\it $(S, \Omega)$-Waldspurger functional} if it satisfies
\begin{equation}\label{Wald-fnal-defn}
\ell(\pi(t)v) = \Omega(t) \ell(v) \qquad \text{ for all } t \in T(F), v \in V.
\end{equation}
If $\pi$ has an $(S, \Omega)$-Waldspurger model then we obtain a $(S, \Omega)$-Waldspurger
functional $\ell$ by $\ell(B) = B(1)$. On the other hand, if $\pi$ has an
$(S, \Omega)$-Waldspurger functional, we obtain an $(S, \Omega)$-Waldspurger model for
$\pi$ by the map $v \mapsto B_v$, where $B_v(g) = \ell(\pi(g)v)$. Observe that a necessary condition for a $(S, \Omega)$-Waldspurger model or functional to exist is that $\Omega |_{F^\times} = \omega_\pi$, the central character of $\pi$.

\subsection{The unramified case}\label{ExplicitValuesUnram}
Throughout this subsection, we suppose that $\pi$ is unramified.
\subsubsection{Preliminaries on the spherical vector in the Waldspurger model}
As $\pi$ is unramified, we have that $\pi = \chi_1 \times \chi_2$ where $\chi_1, \chi_2$  are unramified characters of $F^\times$. Let $\Omega$ be any character of $L^\times$ such that $\Omega |_{F^\times} = \chi_1 \chi_2$. By Saito \cite{Sai} and Tunnell \cite{Tu} or Gross-Prasad \cite{GP} or \cite{FMP}, it is known that $\pi$ has a $(S, \Omega)$-Waldspurger model for any such $\Omega$. Let $B_0$ be the spherical vector in the $(S, \Omega)$-Waldspurger model of $\pi$. Our first task is to give explicit formulas for the values of $B_0(g)$ for all $g \in H(F)$. This is done in the case $\Big(\frac L{\p}\Big) = \pm 1$ and both $c(\Omega) = c(\pi) = 0$ in \cite{BFF}. We will answer this for all $\Omega$ and also for $\Big(\frac L{\p}\Big)=0$. Also, our methods are different from those of \cite{BFF}.  The assumptions on the torus gives the following useful decomposition (see \cite{Su})
\begin{equation}\label{GL2-decomp}
H(F) = \bigsqcup\limits_{m \geq 0} T(F) \mat{\varpi^m}{}{}{1} K.
\end{equation}
Since $B_0$ is the spherical vector in a $(S, \Omega)$-Waldspurger model, we see that $B_0$ is completely determined by its values on $\mat{\varpi^m}{}{}{1}$ with $m\geq 0$. We have the following vanishing result depending on $c(\Omega)$.
\begin{lemma}\label{vanishing-lemma}
Let $c(\Omega) > 0$. Then for all $0 \leq m < c(\Omega)$, we have
$$B_0(\mat{\varpi^m}{}{}{1}) = 0.$$
\end{lemma}
\begin{proof}
Let $t(x,y) \in (1 + \P^m)\cap\OF_L^\times$ be such that $\Omega(t(x,y)) \neq 1$. Note that this implies that $x+\bb y/2+\cc y \beta \in (1 + \P^m)\cap\OF_L^\times$, which means that $y \in \p^m$. Hence, we have
\begin{align*}
B_0(\mat{\varpi^m}{}{}{1}) &= \Omega(t(x,y))^{-1} B_0(t(x,y)\mat{\varpi^m}{}{}{1}) \\
&= \Omega(t(x,y))^{-1} B_0(\mat{\varpi^m}{}{}{1} \underbrace{\mat{\varpi^{-m}}{}{}{1}t(x,y)\mat{\varpi^m}{}{}{1}}_{\in K}) \\
&= \Omega(t(x,y))^{-1} B_0(\mat{\varpi^m}{}{}{1}),
\end{align*}
which completes the proof.
\end{proof}

\subsubsection{Hecke operator}
The spherical vector is an eigenfunction of the Hecke operator $T(\varpi)$, which corresponds to the characteristic function of the double coset $K\mat{\varpi}{}{}{1} K$ in the Hecke algebra of $K$-bi-invariant functions on $H(F)$. We have the following eigenvalue relation
\begin{equation}\label{Hecke-e-val}
T(\varpi) B_0 = \lambda B_0, \qquad \qquad \lambda = q^{1/2}\big(\chi_1(\varpi)+\chi_2(\varpi)\big).
\end{equation}
Note that the above eigenvalue can be easily checked by using the coset decomposition below and applying $T(\varpi)$ to the spherical vector in the induced model of $\pi$. We have the following decomposition of the double coset into a disjoint union of single cosets.
$$K\mat{\varpi}{}{}{1} K = \bigsqcup\limits_{u \in \OF/\p} \mat{\varpi}{u}{}{1} K \sqcup \mat{1}{}{}{\varpi} K.$$
Hence, we get the key relation to obtain the explicit values of $B_0$. For all $g \in H(F)$, we have
\begin{equation}\label{Hecke-cond-eqn}
\sum\limits_{u \in \OF/\p} B_0(g\mat{\varpi}{u}{}{1}) + B_0(g \mat{1}{}{}{\varpi}) = \lambda B_0(g), \qquad \lambda = q^{1/2}\big(\chi_1(\varpi)+\chi_2(\varpi)\big).
\end{equation}
We wish to use the above equation with $g = \mat{\varpi^m}{}{}{1}$. As will be clear, the case $m=0$ is the most complicated and uses a lot of information regarding the underlying number theory. Of course, that case occurs only if $\Omega$ is also unramified. \vskip 0.2in

\begin{lemma}\label{case1-lemma}
We have
\begin{equation}\label{case1-m=0-eqn}
B_0(\mat{\varpi^m}{}{}{1} \mat{1}{}{}{\varpi}) = \begin{cases} \omega_\pi(\varpi) B_0(\mat{\varpi^{m-1}}{}{}{1}) & \text{ if } m > 0;\\
B_0(\mat{\varpi}{}{}{1}) & \text{ if } m = 0, \aaa \in \OF^\times; \\
\Omega(\varpi_L) B_0(1) & \text{ if } m = 0, \aaa \in \p, \Big(\frac L{\p}\Big) = 0;\\
\Omega(1, \varpi) B_0(1) & \text{ if } m = 0, \aaa \in \p, \Big(\frac L{\p}\Big) = 1.\end{cases}
\end{equation}
\end{lemma}
\begin{proof}
The $m > 0$ case is clear. Let $m=0$. By Lemma \ref{vanishing-lemma}, we can assume that $c(\Omega) = 0$. Let $\aaa \in \OF^\times$. We have the matrix identity
$$t(x,y) \mat{1}{}{}{\varpi} = \mat{\varpi}{}{}{1} \underbrace{\mat{}{\cc}{-\aaa}{\bb \varpi}}_{\in K}, \qquad t(x,y) = \mat{}{\cc}{-\aaa}{\bb}.$$
Note that, in this case, we have $t(x,y) \in \OF_L^\times$ and hence $\Omega(t(x,y)) = 1$. This gives us the $m = 0, \aaa \in \OF^\times$ case. 

Now let $\aaa \in \p$. By Lemma 2.1 of \cite{FMP}, we see that this implies $\Big(\frac L{\p}\Big) = 0$ or $1$. First let $\Big(\frac L{\p}\Big) = 0$. Again by Lemma 2.1 of \cite{FMP}, we have $v_L(\beta) = v(\aaa) = 1$, which implies that $\bb \in \p$.  We have the matrix identity
$$t(x,y) \mat{1}{}{}{\varpi}  = \mat{}{\cc}{-\frac{\aaa}{\varpi}}{-\bb} \in K, \qquad \text{ with } x = -\frac{\bb}{2 \varpi}, y = \frac 1{\varpi}.$$
We have $\Omega(t(x,y)) = \Omega(\varpi^{-1}(-\bb+\cc\beta)) = \omega_\pi(\varpi)^{-1} \Omega(\bar\beta) = \omega_\pi(\varpi)^{-1} \Omega(\varpi_L) = \Omega(\varpi_L)^{-1}$. Here, we have again used that $c(\Omega) =0$. Hence, we get the $m=0, \aaa \in \p, \Big(\frac L{\p}\Big) = 0$ case. 

Now, let $\Big(\frac L{\p}\Big) = 1$. Since $\dd \in \OF^\times$ and $\frac{\bb+\sqrt{\dd}}{2\cc} \frac{\bb-\sqrt{\dd}}{2\cc} = \frac{\aaa}{\cc}$, we have $v(\frac{\bb-\sqrt{\dd}}{2\cc})=v(\aaa)$. If $v(\aaa)=1$, then the same matrix identity as above is valid. In this case $\Omega(t(x,y)) = \omega_\pi(\varpi)^{-1} \Omega(\bar\beta) = \omega_\pi(\varpi)^{-1} \Omega(\frac{\bb-\sqrt{\dd}}{2\cc}, \frac{\bb+\sqrt{\dd}}{2\cc}) = \omega_\pi(\varpi)^{-1} \Omega(\varpi,1) =  \Omega(1,\varpi)^{-1}$. If $v(\aaa) > 1$, then we have the matrix identity
$$t(x,y) \mat{1}{}{}{\varpi} = \mat{1}{\cc}{-\aaa/\varpi}{-\bb+\varpi} \in K, \qquad \text{ with }  x = 1-\frac{\bb}{2 \varpi}, y = \frac 1{\varpi}.$$
In this case, 
$\Omega(t(x,y)) = \Omega(\varpi^{-1}(\varpi-\cc\bar\beta)) = \omega_\pi(\varpi)^{-1} \Omega(\varpi-\cc\frac{\bb-\sqrt{\dd}}{2\cc}, \varpi-\cc\frac{\bb+\sqrt{\dd}}{2\cc}) = \omega_\pi(\varpi)^{-1} \Omega(\varpi,1) = \Omega(1, \varpi)^{-1}$, since $\frac{\bb+\sqrt{\dd}}{2\cc} \in \OF^\times$ and $v(\frac{\bb-\sqrt{\dd}}{2\cc})=v(\aaa)>1$. This completes the proof of the lemma.
\end{proof}

\begin{lemma}\label{case2-lemma}
Let $u \in (\OF/\p)^\times, m \geq c(\Omega)$. We have
\begin{equation}\label{case2-eqn}
B_0(\mat{\varpi^m}{}{}{1}\mat{\varpi}{u}{}{1}) = \begin{cases} 
\Omega(\varpi_L) B_0(1) & \text{ if } m=0, \Big(\frac L{\p}\Big)=0, \aaa \in \OF^\times, u = u_0;\\
\Omega(\varpi,1) B_0(1) & \text{ if } m=0, \Big(\frac L{\p}\Big)=1, \aaa \in \OF^\times, u = (-\bb+\sqrt{\dd})/(2\aaa);\\
\Omega(1,\varpi) B_0(1) & \text{ if } m=0, \Big(\frac L{\p}\Big)=1, \aaa \in \OF^\times, u = (-\bb-\sqrt{\dd})/(2\aaa);\\
\Omega(\varpi,1) B_0(1) & \text{ if } m=0, \Big(\frac L{\p}\Big)=1, \aaa \in \p, u = -\cc/\bb;\\
B_0(\mat{\varpi^{m+1}}{}{}{1}) & \text{ otherwise.}
\end{cases}
\end{equation}
Here, in the $\Big(\frac L{\p}\Big)=0$ case $u_0$ is the unique element of $\OF/\p$ such that $u_0 + \beta \not\in \OF_L^\times$. 
\end{lemma}
\begin{proof}
For $u \in (\OF/\p)^\times$ and $m \geq 0$, set $\alpha_{u,m} := \cc + \bb \varpi^m u + \aaa \varpi^{2m}u^2$. First assume that $\alpha_{u,m} \in \OF^\times$. Then we have the matrix identity
$$t(x,y) \mat{\varpi^m}{}{}{1}\mat{\varpi}{u}{}{1} = \mat{\varpi^{m+1}}{}{}{1} \underbrace{\mat{1}{}{\frac{\aaa\varpi^{2m+1}u}{\cc}}{\frac{\alpha_{u,m}}{\cc}}}_{\in K}, \qquad \text{ with } x = 1+\frac{\bb u \varpi^m}{2 \cc}, y \in -\frac{u \varpi^m}{\cc}.$$
Note that, in this case, $t(x,y) = 1+u\varpi^m\bar\beta \in 1+\P^{c(\Omega)}$, since $m \geq c(\Omega)$. Hence, $\Omega(t(x,y))=1$. 

Now, suppose that $\alpha_{u,m} \in \p$. This implies that $m=0$ and $\Big(\frac L{\p}\Big)=0,1$. Hence, $c(\Omega) = 0$. First assume that $\Big(\frac L{\p}\Big)=0$. If $\aaa \in \p$, then $\bb \in \p$ and hence $\alpha_{u,0} \in \OF^\times$ for all $u \in (\OF/\p)^\times$. So we are in the previous case. If $\aaa \in \OF^\times$, then there is a unique $u_0 \in (\OF/\p)^\times$ such that $\alpha_{u_0,0} \in \varpi \OF^\times$. Hence, $\aaa + \bb (\aaa u_0/\cc) + \cc (\aaa u_0/\cc)^2 \in \varpi\OF^\times$, which implies that $v_L(\aaa u_0/\cc + \beta) = 1$. We have the following matrix identity
$$t(x,y) \mat{\varpi}{u_0}{}{1} = \mat{1+\frac{\bb}{\aaa u_0}}{\frac{\alpha_{u_0,0}}{\aaa \varpi u_0}}{-u_0^{-1}}{} \in K, \qquad \text{ with } y = \frac 1{\aaa \varpi u_0}, x = \varpi^{-1}+\bb y/2.$$
Note that $\Omega(t(x,y)) = \Omega(\varpi^{-1}+\bb y/2 + y\sqrt{\dd}/2) = \Omega(\varpi^{-1}+\cc/(\varpi \aaa u_0) \beta) = \omega_\pi(\varpi)^{-1} \Omega(\aaa u_0/\cc + \beta) = \omega_\pi(\varpi)^{-1} \Omega(\varpi_L) = \Omega(\varpi_L)^{-1}$. The other cases are computed similarly. 
\end{proof}

\subsubsection{Values of the spherical vector in the Waldspurger model}
We have the following result for the explicit values of $B_0$.
\begin{proposition}\label{main-result-prop}
Let $\pi = \chi_1 \times \chi_2$ with $\chi_1 \chi_2^{-1} \neq | \,|^{\pm 1}$ and $\chi_1, \chi_2$ unramified. Let $\Omega$ be a character of $L^\times$ such that $\Omega |_{F^\times} = \omega_\pi$ and $c(\Omega)$ as defined in (\ref{conductor-of-Omega}). Let $\pi$ be given by its $(S, \Omega)$-Waldspurger model and let $B_0$ be a spherical vector in $\pi$. Let
$$R(x):= \sum\limits_{m \geq c(\Omega)} B_0(\mat{\varpi^m}{}{}{1}) x^m$$ 
be a formal power series. Let $\lambda = q^{1/2}\big(\chi_1(\varpi)+\chi_2(\varpi)\big)$. Then we have the following formula
\begin{equation}\label{main-result-formula}
R(x) = \frac{(q-\kappa x)x^{c(\Omega)}}{\omega_\pi(\varpi)x^2- \lambda x + q}B_0(\mat{\varpi^{c(\Omega)}}{}{}{1}) ,
\end{equation}
where
\begin{equation}\label{num-formula}
\kappa = \begin{cases} 0 & \text{ if } c(\Omega) >  0;\\
\frac{\lambda}{q+1} & \text{ if } c(\Omega)=0, \Big(\frac L{\p}\Big) = -1;\\
\Omega(\varpi_L) & \text{ if } c(\Omega)=0, \Big(\frac L{\p}\Big) = 0;\\
-\frac{\lambda}{q-1} + \frac q{q-1}(\Omega(\varpi,1)+\Omega(1,\varpi)) & \text{ if } c(\Omega)=0, \Big(\frac L{\p}\Big) = 1.
\end{cases}
\end{equation}
\end{proposition}
\begin{proof}
For $m\geq 0$, we set $A_m = B_0(\mat{\varpi^m}{}{}{1})$. Using (\ref{Hecke-cond-eqn}) with $g = \mat{\varpi^m}{}{}{1}$ and Lemmas \ref{vanishing-lemma},  \ref{case1-lemma}, \ref{case2-lemma}, we get for $m \geq c(\Omega), m > 0$
\begin{equation}\label{recurrence-reln-m>0}
qA_{m+1} + \omega_\pi(\varpi) A_{m-1} = \lambda A_m.
\end{equation}
From this we get the following relation between the generating series.
\begin{equation}\label{gen-ser-reln}
q \sum\limits_{\text{max}(c(\Omega),1)}^\infty A_{m+1} x^m + \omega_\pi(\varpi) \sum\limits_{\text{max}(c(\Omega),1)}^\infty A_{m-1} x^m = \lambda \sum\limits_{\text{max}(c(\Omega),1)}^\infty A_m x^m.
\end{equation}
Let us first consider the case where $c(\Omega) > 0$. Then (\ref{gen-ser-reln}) gives us 
$$\frac qx \sum\limits_{c(\Omega)}^\infty A_{m+1} x^{m+1} + \omega_\pi(\varpi) x \sum\limits_{c(\Omega)}^\infty A_{m-1} x^{m-1} = \lambda \sum\limits_{c(\Omega)}^\infty A_m x^m, $$
which implies
$$  q\big(R(x) - A_{c(\Omega)} x^{c(\Omega)}\big) + \omega_\pi(\varpi) x^2 R(x) = \lambda x R(x).
$$
Solving for $R(x)$ we get the $c(\Omega) > 0$ case of the proposition.

Next, let $c(\Omega) = 0$. We get the following relation from (\ref{gen-ser-reln})
$$\frac qx \sum\limits_2^\infty A_m x^m + \omega_\pi(\varpi) x \sum\limits_0^\infty A_m x^m = \lambda \sum\limits_1^\infty A_m x^m.$$
Hence, we get
$$q\big(R(x) - A_0 - A_1 x\big) + \omega_\pi(\varpi) x^2 R(x) = \lambda x \big(R(x) - A_0\big).$$
Solving for $R(x)$ we have
$$R(x) = \frac{qA_0+qA_1x-\lambda A_0 x}{\omega_\pi(\varpi) x^2-\lambda x + q}.$$
We obtain the following information regarding the above numerator from Lemmas \ref{case1-lemma}, \ref{case2-lemma} using (\ref{Hecke-cond-eqn}) with $g=1$.
\begin{description}
\item[$\Big(\frac L{\p}\Big) = -1$:] We have $(q+1) A_1 = \lambda A_0$. Hence
$$qA_0+qA_1x-\lambda A_0 x = A_0(q-\frac{\lambda}{q+1}x).$$

\item[$\Big(\frac L{\p}\Big) = 0$:] We have $qA_1 + \Omega(\varpi_L)A_0 = \lambda A_0$. Hence
$$qA_0+qA_1x-\lambda A_0 x = A_0(q-\Omega(\varpi_L)x).$$

\item[$\Big(\frac L{\p}\Big) = 1$:] We have $(q-1)A_1 + \big(\Omega(\varpi,1) + \Omega(1,\varpi)\big) A_0 = \lambda A_0$. Hence
$$qA_0+qA_1x-\lambda A_0 x = A_0(q - (-\frac{\lambda}{q-1} + \frac q{q-1}(\Omega(\varpi,1)+\Omega(1,\varpi)))x).$$
\end{description}
This competes the proof of the proposition.
\end{proof}

\subsection{Explicit values for an unramified twist of the Steinberg representation}\label{ExplicitValuesSteinberg}

Throughout this section we assume that the representation $\pi$ is an unramified twist of the Steinberg representation, i.e., $\pi=\chi\St_{\GL_2}$, where $\chi$ is an unramified character of $F^{\times}$. We let $\Omega$ be any character of $L^{\times}$ such that $\Omega|_{F^{\times}}=\omega_{\pi}=\chi^2$. For the field case, \cite{Wald85} states that $\pi$ has an $(S,\Omega)$-Waldspurger model if and only if $\Omega\neq\chi\circ N_{L/F}$.  Note, if $B_0$ is a new form  in the $(S,\Omega)$-Waldspurger model of $\pi$, then $B_0$ is right invariant under the Iwahori subgroup $I = \mat{\OF}{\OF}{\p}{\OF} \cap K$, 
\begin{equation}\label{trace-Steinberg}\sum_{u\in\mathfrak{o}/\mathfrak{p}}B_0(g\mat{1}{}{u}{1})=-B_0(gw)\text{, for } w=\mat{0}{1}{-1}{0},\end{equation}
and
\begin{equation}\label{AL-Steinberg}B_0(g\mat{}{1}{\varpi}{})=-\chi(\varpi)B_0(g).\end{equation}
Using (\ref{GL2-decomp}), we have the following double coset decomposition.
\begin{align}\label{Iwahori-GL2-decomp}
H(F) &= \bigsqcup\limits_{m > 0} \Big(T(F) \mat{\varpi^m}{}{}{1} I \sqcup T(F) \mat{\varpi^m}{}{}{1} w I\Big) \\
& \qquad \bigsqcup \begin{cases} T(F) w I & \text{ if } \Big(\frac{L}{\p}\Big) = -1;\\ 
T(F) w I \sqcup T(F) \mat{1}{}{u_0}{1} I & \text{ if } \Big(\frac{L}{\p}\Big) = 0;\\
T(F) w I \sqcup T(F) \mat{1}{}{u_1}{1} I \sqcup T(F) \mat{1}{}{u_2}{1} I & \text{ if } \Big(\frac{L}{\p}\Big) = 1.\end{cases} \nonumber
\end{align}
In the ramified case, $u_0$ is the unique element of $\OF/\p$ such that $\aaa + \bb u_0 + \cc u_0^2 \in \p$. In the split case, $u_1, u_2$ are the two distinct elements of $\OF/\p$ such that $\aaa + \bb u_i + \cc u_i^2 \in \p$. 
We will begin by restating the relevant portions of Lemma 4.4 in \cite{FMP}. 
\begin{lemma}\cite[Lemma 4.4]{FMP}\label{FMPlemma}
Suppose that $B_0$ is a new form in the $(S,\Omega)$-Waldspurger model of $\pi$. Then,
\begin{enumerate}
\item
For $m>0$, we have
$$B_0(\mat{\varpi^{m}}{}{}{1}w)=\frac{\chi(\varpi)^m}{q^{m}}B_0(w).$$
\item
For $m>0$, we have
$$B_0(\mat{\varpi^m}{}{}{1})=\begin{cases}-\frac{\chi(\varpi)^m}{q^{m-1}}B_0(w)&\text{ if }m\geq c(\Omega)\\ 0&\text{ if }m<c(\Omega).\end{cases}$$
\item
If $L/F$ is ramified, then
$$B_0(\mat{1}{}{u_0}{1})=\begin{cases}-q B_0(w)&\text{ if }c(\Omega)=0\\0&\text{ if }c(\Omega)>0.\end{cases}$$

\end{enumerate}
\end{lemma}
We note that this lemma is only stated for fields in \cite{FMP}, but the proof of part i) and part ii) in the split case follows from exactly the same argument.

We will also need the following analogue of part iii) of the previous lemma in the split case. Note, from Thm. 1.6 in \cite{FMP}, we know that $\pi$ always admits an $(S,\Omega)$-Waldspurger model when $L/F$ is split. 

\begin{lemma}\label{split-Steinberg}
Suppose that $\Big(\frac L\p\Big)=1$. Let $u_1,u_2\in\mathfrak{o}$ be inequivalent modulo $\mathfrak{p}$ and satisfy $cu_i^2+bu_i+a\in\mathfrak{p}$ for $i=1,2$. Then,
\begin{enumerate}
\item If $c(\Omega) > 0$, then we have, for $i=1,2$,
$$B_0(\mat{1}{}{u_i}{1}) = 0.$$ 

\item Let $c(\Omega) = 0$. Assume that $\Omega(1,\varpi) = \chi(\varpi)$. Then $B_0(w) = 0$ and 
$$B_0(\mat{1}{}{u_1}{1}) = -B_0(\mat{1}{}{u_2}{1}).$$

\item Let $c(\Omega) = 0$. Assume that $\Omega(1,\varpi) \neq \chi(\varpi)$. Then $B_0(w) \neq 0$ and
$$B_0(\mat{1}{}{u_1}{1})=\frac{q-1}{\chi(\varpi)\Omega(1,\varpi)^{-1}-1}B_0(w), B_0(\mat{1}{}{u_2}{1})=\frac{q-1}{\chi(\varpi)^{-1}\Omega(1,\varpi)-1} B_0(w).$$
\end{enumerate}
\end{lemma}
\begin{proof}
First, set $x=\sqrt{{\bf d}}/2+\varpi$ and $y=1$. Then, one can check that
$$t(x,y)\mat{1}{}{u_2}{1}=\mat{1}{}{u_1}{1}\mat{1}{}{}{\varpi}w\mat{-1}{}{}{1}\mat{-\sqrt{{\bf d}}/{\bf c}}{1}{\varpi}{{\bf c}}.$$
Note, the last two matrices on the right hand side are in $I$. Thus, we have
$$\Omega(\sqrt{{\bf d}}+\varpi,\varpi)B_0(\mat{1}{}{u_2}{1})=B_0(\mat{1}{}{u_1}{1}\mat{1}{}{}{\varpi}w).$$
Also, by (\ref{AL-Steinberg}) we have
$$B_0(\mat{1}{}{u_1}{1}\mat{1}{}{}{\varpi}w)=-\chi(\varpi)B_0(\mat{1}{}{u_1}{1}).$$
Hence, we get
\begin{equation}\label{eqn11}
\Omega(\sqrt{{\bf d}}+\varpi,\varpi)B_0(\mat{1}{}{u_2}{1}) = -\chi(\varpi)B_0(\mat{1}{}{u_1}{1}).
\end{equation}
Now, let us assume that $c(\Omega)>0$. Let $(a_1,a_2)\in\mathfrak{o}^{\times}\oplus\mathfrak{o}^{\times}$ satisfy $\Omega((a_1,a_2))\neq 1$, which is possible since $c(\Omega)>0$. Using this, we set 
$$x=\frac{a_1+a_2}{2},y=\frac{a_1-a_2}{\sqrt{d}}.$$
Then, we have that
$$\mat{1}{}{-u_1}{1}t(x,y)\mat{1}{}{u_1}{1}\in I.$$
Thus,
$$\Omega((a_1,a_2))B(\mat{1}{}{u_1}{1})=B(t(x,y)\mat{1}{}{u_1}{1})=B(\mat{1}{}{u_1}{1}),$$
and since $\Omega((a_1,a_2))\neq 1$, we have $B(\mat{1}{}{u_1}{1})=0$. By (\ref{eqn11}), we also get $B(\mat{1}{}{u_2}{1})=0$. This completes the proof of part i).

Next, if we set $x={\bf b}/2+{\bf c}u$ and $y=1$ for any $u\in\mathfrak{o}/\mathfrak{p}$ with $u$ not equivalent to $u_1$ or $u_2$ modulo $\mathfrak{p}$, then we have
$$\mat{1}{}{u}{1}\mat{-{\bf c}}{{\bf b}+{\bf c}u}{}{-\beta_{u,0}}=t(x,y)w,$$
where $\beta_{u,0}$ is defined in Lemma 3.2 of \cite{P-Steinberg}. From this, it follows that
$$B_0(\mat{1}{}{u}{1})=\Omega(u+\beta)B_0(w).$$
Applying this to (\ref{trace-Steinberg}) we have
$$B_0(\mat{1}{}{u_1}{1})+B_0(\mat{1}{}{u_2}{1})=-B_0(w)\left[\sum_{\substack{u\in\mathfrak{o}/\mathfrak{p}\\u\neq u_1,u_2}}\Omega(u+\beta)+\Omega(1)\right].$$
By  Lemma 3.4 in \cite{P-Steinberg}, the summation on the right hand side is over a complete set of representatives for $\OF_L^\times/(\OF^\times + \P)$, and hence is equal to $q-1$ since $c(\Omega)=0$. So, we get
$$B_0(\mat{1}{}{u_1}{1})+B_0(\mat{1}{}{u_2}{1})=-(q-1)B_0(w).$$
Combining this with (\ref{eqn11}), we have
$$(\chi(\varpi)^{-1}\Omega(1,\varpi)-1)B_0(\mat{1}{}{u_2}{1})=(q-1)B_0(w),$$
where we have used the fact that $\Omega$ is unramified and that $\sqrt{{\bf d}}\in\mathfrak{o}^{\times}$. Parts ii) and iii) now follow.
\end{proof}

When $B_0(w)\neq 0$, we will choose $B_0$ to be normalized so that $B_0(w)=1$. Note, if $\pi$ admits a non-zero $(S,\Omega)$-Waldspurger model and $B_0(w)=0$, then it is necessarily the case that $L/F$ is split and $c(\Omega) = 0$. In that case, we normalize so that
$$B_0(\mat{1}{}{u_1}{1})=-B_0(\mat{1}{}{u_2}{1})=1.$$

\section{Local non-archimedean zeta integral}\label{IntegralCalculation}
In this section, we will compute the local integral (\ref{local-integral}) in the non-archimedean case. 
We will first compute the zeta integral when the $\GL(2)$ representation is unramified. Finally, we will compute the zeta integral in several cases when the $\GL(2)$ representation is an unramified twist of the Steinberg representation.

\subsection{The local unramified integral}
Now, we will compute the local integral \eqref{local-integral} given by
$$Z(s) = \int\limits_{T(F) \backslash H(F)} f(\eta h, s) B(h) dh.$$
The measure is normalized so that 
$$\int\limits_{T(F)\backslash T(F) K} dt = 1.$$
Let us assume that $\Omega_1, \Omega_2$ are unramified characters and $\pi$ is unramified. This implies that $c(\Omega) = 0$. Choose the unramified section $f$ given by
$$f(\mat{u}{v}{}{w}k) = \Omega_1(u) \Omega_2(w) |u/w|_L^{s+1/2}, \text{ for } \mat{u}{v}{}{w} \in B(L) \text{ and } k \in H(\OF_L).$$
Let $B = B_0$ the spherical vector in $\pi$ normalized so that $B_0(1) = 1$. This is possible by Proposition \ref{main-result-prop}. Hence, we have
\begin{align*}
Z(s) &= \sum\limits_{m = 0}^\infty \int\limits_{T(F)\backslash T(F) \mat{\varpi^m}{}{}{1}K} f(\eta h, s) B_0(h) dh \\
&= \sum\limits_{m = 0}^\infty V_m f(\mat{\varpi^m}{}{}{1}, s) B_0(\mat{\varpi^m}{}{}{1}),
\end{align*}
where, by Lemma 3.5.3 of \cite{Fu}, we have
$$V_m := \int\limits_{T(F)\backslash T(F) \mat{\varpi^m}{}{}{1}K} dt = \begin{cases} (1 - \Big(\frac L\p\Big)q^{-1}) q^m & \text{ if } m \geq 1; \\
1 & \text{ if } m = 0.\end{cases}
$$
We have also used that $\mat{\varpi^{-m}}{}{}{1} \eta \mat{\varpi^m}{}{}{1} \in H(\OF_L)$ since $\beta \in \OF_L$. Hence, we get
\begin{align*}
Z(s) &= \sum\limits_{m = 0}^\infty (1 - \Big(\frac L\p\Big)q^{-1}) q^m \Omega_1(\varpi^m) |\varpi^m|_L^{s+1/2} B_0(\mat{\varpi^m}{}{}{1}) + \Big(\frac L\p\Big)q^{-1}\\
&= (1 - \Big(\frac L\p\Big)q^{-1}) \sum\limits_{m = 0}^\infty \big(\Omega_1(\varpi) q^{-2s}\big)^m B_0(\mat{\varpi^m}{}{}{1}) + \Big(\frac L\p\Big)q^{-1}\\
&= (1 - \Big(\frac L\p\Big)q^{-1})  R(\Omega_1(\varpi) q^{-2s}) + \Big(\frac L\p\Big)q^{-1}.
\end{align*}
Using the formula for $R(x)$ from Proposition \ref{main-result-prop}, after some computation, we get the following result.
\begin{theorem}\label{local-unram-int-thm}
Let $\pi, \Omega_1$ and $\Omega_2$ be unramified. We have
\begin{equation}\label{local-unram-int-eqn}
Z(s) = \frac{L(2s+1/2, \pi \times \Omega_1|_{F^\times})}{L(2s+1, \Omega_1 \Omega_2^{-1})},
\end{equation}
where
$$L(s, \Omega_1 \Omega_2^{-1}) = \begin{cases} \big(1- \Omega_1 \Omega_2^{-1}(\varpi)q^{-2s}\big)^{-1} & \text{ if } \Big(\frac L\p\Big) = -1; \\ \big(1- \Omega_1 \Omega_2^{-1}(\varpi_L)q^{-s}\big)^{-1} & \text{ if } \Big(\frac L\p\Big) = 0; \\
\big(1- \Omega_1 \Omega_2^{-1}(\varpi,1)q^{-s}\big)^{-1} \big(1- \Omega_1 \Omega_2^{-1}(1,\varpi)q^{-s}\big)^{-1} & \text{ if } \Big(\frac L\p\Big) = 1.\end{cases}$$
\end{theorem}
\subsection{The local integral for the unramified twist of a Steinberg representation}
We now proceed to the case that $\pi=\chi\St$ is the unramified twist of a Steinberg representation, and we let $B_0$ denote the new-form   in the $(S,\Omega)$-Waldspurger model of $\pi$ which was introduced in Section \ref{ExplicitValuesSteinberg}.
\subsubsection{Preliminaries}

For the calculation of $Z(s)$ we will require certain volume calculations throughout. Note, for a subgroup $K'\subset K$ we set
$$V_{K',m}=\int\limits_{T(F)\backslash T(F)\mat{\varpi^m}{}{}{1}K'}dh,$$
where we have normalized the measure so that $V_{K,0}=1$. In what follows, we set $I\subset K$ to be the Iwahori subgroup.

\begin{lemma}\label{Steinberg-volumes}
For $m\geq 0$ we have
\begin{enumerate}
\item
$$V_{w I,m}=\frac{q^{m+1}(1-\left(\frac L\p\right)q^{-1})}{q+1}.$$
\item
For $m\geq 1$,
$$V_{I,m}=\frac{q^{m}(1-\left(\frac L\p\right)q^{-1})}{q+1}.$$
\item
If $\left(\frac L\p\right)=0$, then
$$V_{\mat{1}{}{u_0}{1} I,0}=\frac{1}{q+1}.$$
\item
If $\left(\frac L\p\right)=1$, then
$$V_{\mat{1}{}{u_1}{1} I,0}=V_{\mat{1}{}{u_2}{1} I,0}=\frac{1}{q+1}.$$
\end{enumerate}
\end{lemma}

\begin{proof}
Parts i) and ii) follow from similar arguments as in the proof of Lemmas 3.7.1, 3.7.2 and 3.7.3 in \cite{PS1}. Parts iii) and iv) follow from part i) by applying Lemma 4.1 in \cite{PS3}.
\end{proof}

Throughout, we will use the following expression for $Z(s)$ which is obtained by applying (\ref{Iwahori-GL2-decomp}),
\begin{equation}\label{zeta-Steinberg}
Z(s)=\sum\limits_{m = 1}^\infty \left(\int\limits_{T(F)\backslash T(F) \mat{\varpi^m}{}{}{1}\I}\!\!\!\!\!\!\!\!\!\!\!\!\!\!\!\!\!\!\!\! f(\eta h, s) B_0(h) dh+\!\!\!\!\!\!\!\!\!\!\!\!\!\!\!\!\!\!\!\!\int\limits_{T(F)\backslash T(F) \mat{\varpi^m}{}{}{1}w\I}\!\!\!\!\!\!\!\!\!\!\!\!\!\!\!\!\!\!\!\! f(\eta h, s) B_0(h) dh\right)+\!\!\!\!\!\int\limits_{T(F)\backslash T(F)K} \!\!\!\!\!f(\eta h, s) B_0(h) dh.\end{equation}
\subsubsection{Integrating against a ramified principal series}

In this section we assume that $c(\Omega_1)=1$ and $c(\Omega_2)=0$, so that $c(\Omega)=1$, which implies that $\pi$ has an $(S,\Omega)$-Waldspurger model. We choose the section $f\in I(\Omega_1,\Omega_2,s)$ given by the formula
$$f(h,s)= \begin{cases} \Omega_1(a)\Omega_2(d)\left|\frac{a}{d}\right|_L^{s+1/2} & \text{ if } h\in\mat{a}{\ast}{}{d}\mat{1}{0}{1}{1}K_1( \P_L),\\
0 & \text{ o.w.} \end{cases}$$
where
$$\mat{a}{\ast}{}{d}\in B(L),$$
and
$$K_1( \P_L)=\left\{\mat{a}{b}{c}{d}\in H(\mathfrak{o}_L):c\in \P_L,d\in 1+ \P_L\right\}.$$
Note, when $L$ is a field, one can easily show that
$$B(L)\mat{1}{0}{1}{1}K_1( \P_L)=B(L)\mat{1}{0}{1}{1}K_1( \mathfrak{p}_L),$$
which justifies our choice of section $f$.

We will need the following lemma for evaluating the zeta integral $Z(s)$.

\begin{lemma}\label{SteinbergIwahoriInvariance}
Let $f$ be as above. Then, $f$ is right invariant with respect to $I$.
\end{lemma}
\begin{proof}
Any element of $I$  can be written as the product of an element of  $Z(\OF^\times)$ and an element in $K_1(\P_L)$. Now, using the relation $(\Omega_1 \Omega_2)|_{F^\times} = \omega_\pi^{-1}$ and the fact that $\omega_\pi$ is unramified, we get the lemma.
%
\end{proof}

Using this lemma, we obtain the following result which simplifies our zeta integral.
\begin{proposition}\label{SteinbergPrinSeries}
With notation as above we have,
\begin{enumerate}
\item
For $m>0$,
$$\int\limits_{T(F)\backslash T(F) \mat{\varpi^m}{}{}{1} I} f(\eta h, s) B_0(h) dh=0.$$
\item
$$\int\limits_{T(F)\backslash T(F)K} f(\eta h, s) B_0(h) dh=V_{w I,0}.$$
\end{enumerate}
\end{proposition}
\begin{proof}
First, we prove i). Applying Lemma \ref{SteinbergIwahoriInvariance}, it is enough to show that 
$$\eta \mat{\varpi^m}{}{}{1}\notin B(L)\mat{1}{0}{1}{1}K_1(\P_L).$$
This follows from the fact that $\eta \mat{\varpi^m}{}{}{1} \in B(L) K_1(\P_L)$ and $B(L) K_1(\P_L) \cap B(L)\mat{1}{0}{1}{1}K_1(\P_L)$ is empty.
In order to prove ii), note that we can rewrite $w$ as
$$w=\mat{-1}{}{}{-1}\mat{1}{-1}{}{1}\mat{1}{}{1}{1}\mat{1}{-1}{}{1},$$
from which it follows that $\eta w = w (w^{-1}\eta w)$ is in the support of $f$ giving us
$$\int\limits_{T(F)\backslash T(F)w I} f(\eta h, s) B_0(h) dh=V_{w I,0}.$$
Note that we have used $c(\Omega)>0$ in the split case to get $B_0(w) \neq 0$. By  Lemmas \ref{FMPlemma} and \ref{split-Steinberg}, we have $B_0(\mat{1}{}{u_0}{1})=0$ if $L/F$ is a ramified field extension and $B_0(\mat{1}{}{u_i}{1})=0, i=1,2$ if $L/F$ is a split extension. Now part ii) follows from (\ref{Iwahori-GL2-decomp}).
\end{proof}

With this proposition in hand, we are now prepared to compute the zeta integral, i.e., 
\begin{eqnarray*}
Z(s)&=&\sum\limits_{m = 1}^\infty \left(\int\limits_{T(F)\backslash T(F) \mat{\varpi^m}{}{}{1} I}\!\!\!\!\!\!\!\!\!\!\!\!\!\!\!\!\!\!\!\! f(\eta h, s) B_0(h) dh+\!\!\!\!\!\!\!\!\!\!\!\!\!\!\!\!\!\!\!\!\int\limits_{T(F)\backslash T(F) \mat{\varpi^m}{}{}{1}w I}\!\!\!\!\!\!\!\!\!\!\!\!\!\!\!\!\!\!\!\! f(\eta h, s) B_0(h) dh\right)+\!\!\!\!\!\int\limits_{T(F)\backslash T(F)K} \!\!\!\!\!f(\eta h, s) B_0(h) dh\\
&=&V_{w I,0}+\sum\limits_{m = 1}^\infty\int\limits_{T(F)\backslash T(F) \mat{\varpi^m}{}{}{1}w I} f(\eta h, s) B_0(h) dh\\
&=&\sum\limits_{m = 0}^\infty V_{w I,m}f(\mat{-\varpi^m}{}{}{-1}\mat{}{1}{-1}{})B_0(\mat{\varpi^m}{}{}{1}w)\\
&=&\sum\limits_{m = 0}^\infty \frac{q^m(q-\left(\frac L\p\right))}{q+1}\Omega_1(\varpi)^m|\varpi^m|_L^{s+1/2}\chi(\varpi)^{m}q^{-m}\\
&=&\frac{q-\left(\frac L\p\right)}{q+1}\sum\limits_{m = 0}^\infty (\Omega_1(\varpi)\chi(\varpi)q^{-2s-1})^m.
\end{eqnarray*}

With this calculation, we have shown the following theorem.
\begin{theorem}\label{stein-ram-int-thm}
Let $\pi=\chi\St_{\GL_2}$ with $\chi$ an unramified character of $F^{\times}$. Let $\Omega_1$ and $\Omega_2$ be characters of $L^{\times}$ with $c(\Omega_1)=1$ and $\Omega_2$ being unramified, and suppose that $\Omega_1\Omega_2|_{F^{\times}}=\omega_{\pi}^{-1}$. Then,
$$Z(s)=\frac{q-\left(\frac L\p\right)}{q+1}\frac{L(2s+1/2,\pi\times\Omega_1|_{F^{\times}})}{L(2s+1,\Omega_1\Omega_2^{-1})}.$$
\end{theorem}

\subsubsection{Integrating against an unramified principal series}

In this section we consider the case that the characters $\Omega_1$ and $\Omega_2$ are unramified, and $I(\Omega_1,\Omega_2,s)$ is irreducible.

In this setting, the condition $\Omega|_{F^{\times}}=\chi^2$ gives that $\Omega=\chi\circ N_{L/F}$ when $L/F$ is unramified, hence $\pi$ does not have an $(S,\Omega)$-Waldspurger model. Similarly, when $L/F$ is ramified we have that $\Omega=\chi'\chi\circ N_{L/F}$ where $\chi'$ is either trivial or the unique unramified quadratic character. It is only in the latter case that $\pi$ has an $(S,\Omega)$-Waldspurger model. Finally, when $L/F$ is split, we simply apply Thm. 1.6 from \cite{FMP} to see that $\pi$ has an $(S,\Omega)$-Waldspurger model.

If we choose the unramified section
\begin{equation}\label{unramnew}
f(\mat{a}{\ast}{}{d}k) = \Omega_1(a) \Omega_2(d) |a/d|_L^{s+1/2}, \text{ for } \mat{a}{\ast}{}{d} \in B(L) \text{ and } k \in H(\OF_L),\end{equation}
then an inner $K$-integral gives a vector in $\pi$ that is spherical, which is impossible. Hence, for that choice of $f$, we have $Z(s) = 0$.

Alternatively, considering the same $\Omega_1,\Omega_2$, we also calculate the local zeta integral by choosing the following section
$$\hat{f}(bk)=f(bkg,s),\text{ for }b\in B(L),k\in\GL_2(\OF_L),g=\mat{\varpi_L^{-1}}{}{}{1},$$
where $f$ is the section from (\ref{unramnew}). Note, this is an old vector in $I(\Omega_1,\Omega_2,s)$.

We present the following calculation, which will be needed to evaluate the zeta integral.
\begin{proposition}
Let $f$ be as above. Then, $\int\limits_{T(F)\backslash T(F)K}\hat{f}(\eta h, s) B_0(h) dh$ is equal to
$$\left\{\begin{array}{ll}\frac{-\Omega_1(\varpi_L)^{-1}q^{s+3/2}}{q+1}(1-\Omega_1(\varpi_L)\Omega_2(\varpi_L)^{-1}q^{-2s-1})&\text{if }\Big(\frac L\p\Big)=0\\
\frac{-\Omega_1(\varpi,1)^{-1}q^{s+1/2}}{q+1}(1-\Omega_1(\varpi,1)\Omega_2(\varpi,1)^{-1}q^{-2s-1})&\text{if }\Big(\frac L\p\Big)=1\text{ and }B_0(w)=0\\
\frac{-(q-1)\Omega_1(\varpi,1)^{-1}q^{s+1/2}}{(q+1)(1-\chi(\varpi)^{-1}\Omega(1,\varpi))}(1-\Omega_1(\varpi,1)\Omega_2(\varpi,1)^{-1}q^{-2s-1})&\text{if }\Big(\frac L\p\Big)=1\text{ and }B_0(w)=1\end{array}\right\}.$$
\end{proposition}
\begin{proof}
Suppose that $L/F$ is ramified. Note, in this case we have $B_0(\mat{1}{}{u_0}{1})=-q$ by Lemma \ref{FMPlemma}. Furthermore, the integral breaks up as
$$\int\limits_{T(F)\backslash T(F)K}\hat{f}(\eta h, s) B_0(h)dh=\int\limits_{T(F)\backslash T(F)w I}\hat{f}(\eta h, s) B_0(h)dh+\!\!\!\!\!\!\!\!\!\!\int\limits_{T(F)\backslash T(F)\mat{1}{}{u_0}{1} I}\hat{f}(\eta h, s) B_0(h)dh.$$
In order to compute the first integral, we use that $\hat{f}$ is right invariant under $ I$, which follows from
$$ I\subseteq\mat{\varpi_L^{-1}}{}{}{1}H(\OF_L)\mat{\varpi_L}{}{}{1},$$
We also need the matrix identity
\begin{equation}\label{etaw-ident}
\eta w=\mat{\varpi_L}{}{}{\varpi_{L}^{-1}}\mat{\varpi_L^{-1}}{}{}{1}\mat{}{1}{-1}{\beta\varpi_L}\mat{\varpi_L}{}{}{1}.\end{equation}
Combining with the volume calculation in Lemma \ref{Steinberg-volumes} we have
$$\int\limits_{T(F)\backslash T(F)w I}\hat{f}(\eta h, s) B_0(h)dh=V_{w I,0}\Omega_2(\varpi_L)^{-1}|\varpi_L|_L^{s+1/2}B_0(w)
=\frac{q^{-s+1/2}}{q+1}\Omega_2(\varpi_L)^{-1}.
$$
The evaluation of the second integral follows by noting that 
$$\eta\mat{1}{}{u_0}{1}=\mat{\varpi_L^{-1}}{}{}{1}\mat{1}{}{\varpi_{L}^{-1}(\beta+u_0)}{1}\mat{\varpi_L}{}{}{1}\in \mat{\varpi_L^{-1}}{}{}{1}H(\OF_L)\mat{\varpi_L}{}{}{1}.$$
From which we obtain
$$\int\limits_{T(F)\backslash T(F)\mat{1}{}{u_0}{1} I}\hat{f}(\eta h, s) B_0(h)dh=\frac{-q^{s+3/2}}{q+1}\Omega_1(\varpi_L)^{-1}.$$
Combining, we have
$$\int\limits_{T(F)\backslash T(F)K}\hat{f}(\eta h, s) B_0(h)dh=\frac{-q^{s+3/2}\Omega_1(\varpi_L)^{-1}}{q+1}(1-\Omega_1(\varpi_L)\Omega_2(\varpi_L)^{-1}q^{-2s-1}).$$
The split case is computed in a similar way.
\end{proof}

In order to calculate $Z(s)$ we also need the following integral calculations, which we calculate using Lemma \ref{Steinberg-volumes} and Lemma \ref{FMPlemma},
$$\int\limits_{T(F)\backslash T(F) \mat{\varpi^m}{}{}{1} I}\!\!\!\!\!\!\!\!\!\!\!\!\!\!\!\!\!\!\!\! \hat{f}(\eta h, s) B_0(h) dh=\frac{-(1-\Big(\frac L\p\Big)q^{-1})q}{q+1}\Omega_1(\varpi_L)^{-1}q^{s+1/2}(\Omega_1(\varpi)\chi(\varpi)q^{-2s-1})^mB_0(w),$$
$$\int\limits_{T(F)\backslash T(F) \mat{\varpi^m}{}{}{1}w I}\!\!\!\!\!\!\!\!\!\!\!\!\!\!\!\!\!\!\!\! \hat{f}(\eta h, s) B_0(h) dh=\frac{(1-\Big(\frac L\p\Big)q^{-1})q}{q+1}\Omega_2(\varpi_L)^{-1}q^{-s-1/2}(\Omega_1(\varpi)\chi(\varpi)q^{-2s-1})^mB_0(w).$$
Combining this with the previous proposition we have the following theorem.
\begin{theorem}\label{stein-unramif-int-thm}
Let $\pi=\chi\St_{\GL_2}$ with $\chi$ an unramified character of $F^{\times}$. Let $\Omega_1$, $\Omega_2$, and $\hat{f}$ be as above. Then,
$$Z(s) = -\frac{q^{s+\frac 12}}{q+1} \frac{L(2s+1/2,\pi\times\Omega_1|_{F^{\times}})}{L(2s+1,\Omega_1\Omega_2^{-1})} \times \begin{cases}q \Omega_1(\varpi_L)^{-1} & \text{ if } \Big(\frac L\p\Big)=0;\\
\Omega_1(\varpi,1)^{-1} &\text{ if }\Big(\frac L\p\Big)=1\text{ and }B_0(w)=0;\\
\frac{(q-1) \Omega_1(\varpi,1)^{-1}}{1-\chi(\varpi)^{-1}\Omega(1,\varpi)} &\text{ if }\Big(\frac L\p\Big)=1\text{ and }B_0(w)=1.
\end{cases}$$
%
\end{theorem}

\section{Local archimedean case}\label{ExplicitValuesArch}
Now, let $F = \R$. Let $K = \SO(2)$ be the maximal compact subgroup of $H(\R)$. For $\ell \geq 1$, let $\pi$ be the discrete series representation ${\mathcal D}_\mu(\ell)$. This representation has the lowest non-negative weight $\ell$ and central character $\mat{u}{}{}{u} \mapsto \sgn(u)^{\ell}|u|^\mu$. We need to obtain the values of a weight $\ell$ vector $B_0$ in the Waldspurger model of ${\mathcal D}_\mu(\ell)$. We obtain a differential equation satisfied by $B_0$ using the fact that $B_0$ is annihilated by the lowering operator. For this, first recall that the Lie algebra ${\mathfrak g} = {\mathfrak sl}(2,\R)$ of $\SL(2,\R)$ is generated by
 $$D = \mat{1}{}{}{-1}, \qquad E = \mat{}{1}{}{}, \qquad F = \mat{}{}{1}{},$$
 and the lowering operator $L$ is an element of the complexified Lie algebra ${\mathfrak g}_\C$ and is defined by 
 \begin{equation}\label{L-operator-defn}
 L = \frac 12 \mat{1}{-i}{-i}{1} = \frac 12 \big(D - iE - iF\big).
 \end{equation}
 An element $X \in {\mathfrak g}$ acts on $B_0$ by
 \begin{equation}\label{Lie-alg-action}
 (X.B_0)(g) = \frac d{dt} \Big|_{t=0} B_0(g \exp(tX)).
 \end{equation}
 We will follow the ideas from \cite{PS2}. We will consider two special cases here corresponding to $S = \pm \mat{1}{}{}{1}$ and $S = \pm \mat{1}{}{}{-1}$, the non-split and split case respectively. 
\subsection{The non-split case}
Let $S = \pm \mat{1}{}{}{1}$. Then
$$T(\R) = \{ \mat{x}{y}{-y}{x} : x, y \in \R, x^2+y^2 \neq 0\} \simeq \C^\times$$
by the isomorphism $\mat{x}{y}{-y}{x} \mapsto x+iy$. We see that any element of $t \in T(\R)$ can be written as 
$$t = \mat{\gamma}{}{}{\gamma} r(\delta), \qquad \text{ where } \gamma > 0, \qquad r(\delta) = \mat{\cos(\delta)}{\sin(\delta)}{-\sin(\delta)}{\cos(\delta)}, \text{ with } \delta \in \R.$$
Let $\Omega$ be a character of $\C^\times$ given by
\begin{equation}\label{Omega-defn-arch-ns}
\Omega(\mat{\gamma}{}{}{\gamma} r(\delta)) = \gamma^\mu e^{im\delta},
\end{equation}
for some $m \in \Z$. Notice that we want $\Omega |_{\R^\times} = \omega_\pi$, and hence we must have $m \equiv \ell \pmod{2}$. 
By the Cartan decomposition, we have (see (18) of \cite{PS2})
\begin{align*}
\GL(2,\R) &= \GL(2,\R)^+ \sqcup \mat{-1}{}{}{1} \GL(2,\R)^+ \\
&= T(\R) \{\mat{\zeta}{}{}{\zeta^{-1}}: \zeta \geq 1\} \SO(2) \sqcup T(\R) \mat{-1}{}{}{1} \{\mat{\zeta}{}{}{\zeta^{-1}}: \zeta \geq 1\}  \SO(2).
\end{align*}
Let us assume that $\pi = {\mathcal D}_\mu(\ell)$ is given by its $(\Omega,S)$-Waldspurger model. Let $B_0 \in \pi$ be weight $\ell$ vector. Hence, we have
$$B_0(tgr(\theta)) = \Omega(t) e^{i\ell \theta} B_0(g).$$
If $B_0(1) \neq 0$, then we get the necessary condition that $m=\ell$. If $B_0(\mat{-1}{}{}{1}) \neq 0$, then, using $\mat{-1}{}{}{1} r(\delta) \mat{-1}{}{}{1} = r(-\delta)$, we get the necessary condition that $m = -\ell$.  In the first case, support of $B_0$ is contained in $\GL(2,\R)^+$ and in the latter case, the support of $B_0$ is contained in $\mat{-1}{}{}{1} \GL(2,\R)^+$. Let us first consider the case $m = \ell$. Let us set $f(\zeta) := B_0(\mat{\zeta}{}{}{\zeta^{-1}})$ for $\zeta \geq 1$. We wish to obtain the action of $L$ on $B_0$. For this, suppose
 $$\mat{\zeta}{}{}{\zeta^{-1}} \exp(tX) = \mat{\gamma(t)}{}{}{\gamma(t)} r(\delta(t)) \mat{\zeta(t)}{}{}{\zeta(t)^{-1}} r(\theta(t)),$$
 where $\gamma(t), \delta(t), \zeta(t)$ and $\theta(t)$ are smooth functions with $\gamma(0) = 1, \delta(0) = \theta(0) = 0$ and $\zeta(t) = \zeta$. Then
 \begin{align}
 (X.B_0)(\mat{\zeta}{}{}{\zeta^{-1}}) &= \frac d{dt} \Big|_{t=0} B_0(\mat{\gamma(t)}{}{}{\gamma(t)} r(\delta(t)) \mat{\zeta(t)}{}{}{\zeta(t)^{-1}} r(\theta(t))) \nonumber \\
 &= \frac d{dt} \Big|_{t=0} \gamma(t)^\mu e^{i\ell (\theta(t)+\delta(t))} f(\zeta(t)) \nonumber \\
 &= \Big(\mu \gamma'(0)+i \ell (\theta'(0) + \delta'(0))\Big) f(\zeta) + \zeta'(0) f'(\zeta). \label{X-action}
 \end{align}
Hence, we need to find the values of the derivatives at $0$ of the above functions for $X = D, E, F$. 
\vskip 0.2in
\noindent{\bf $X = D$ case}:  Let $X=D$. Then $\exp(tD) = \mat{e^t}{}{}{e^{-t}}$. Hence, $\gamma(t) = 1, \delta(t) = \theta(t) = 0$ for all $t$ and $\zeta(t) = \zeta e^t$. Hence \eqref{X-action} gives us
\begin{equation}\label{D-action}
(D.B_0)(\mat{\zeta}{}{}{\zeta^{-1}}) = \zeta f'(\zeta).
\end{equation}
\vskip 0.2in
\noindent{\bf $X = E$ case}:  Let $X=E$. Then $\exp(tE) = \mat{1}{t}{}{1}$. Hence
\begin{equation}\label{some-eqn}
\mat{\zeta}{}{}{\zeta^{-1}} \mat{1}{t}{}{1} = r(\delta(t)) \mat{\zeta(t)}{}{}{\zeta(t)^{-1}} r(\theta(t)).
\end{equation}
We recall the following lemma and proof from the expanded version of \cite{PSS}. 
\begin{lemma}\label{arch-toric-decomp}
 Let $h = \mat{y}{x}{}{y^{-1}}$ with $y \neq 0$. Then $h = k_1 \mat{z}{}{}{z^{-1}} k_2$, with $k_1, k_2 \in \SO(2)$ and
 $$
  z^2 = \frac{1+x^2y^2+y^4 + \sqrt{(1+x^2y^2+y^4)^2-4y^4}}{2y^2}.
 $$
\end{lemma}
\begin{proof}
We may assume that $x\neq0$. By the Cartan decomposition of $\SL_2(\R)$, there exist $k_1, k_2 \in \SO(2)$ and $z > 1$ such that $h = k_1 \mat{z}{}{}{z^{-1}} k_2$. Write $k_1 = \mat{\cos(\delta)}{\sin(\delta)}{-\sin(\delta)}{\cos(\delta)}$ for $\delta \in [0, 2\pi)$. Applying both sides of $h = k_1 \mat{z}{}{}{z^{-1}} k_2$ to $i$ as fractional linear transformations, and using that $\SO(2)$ stabilizes $i$, we get
$$
 y^2i+xy = \frac{\cos(\delta)z^2 i + \sin(\delta)}{-\sin(\delta)z^2 i + \cos(\delta)}.
$$
Simplifying and comparing the coefficients of $i$ and the constant terms, we get
$$
 -z^2xy \sin(\delta)  = \cos(\delta) (z^2-y^2), \qquad (1-z^2y^2) \sin(\delta) = xy \cos(\delta).
$$
Note that, since $x, y \neq 0$, we have $\sin(\delta), \cos(\delta) \neq 0$ and $y \neq \pm z, \pm 1/z$. Hence, we can divide the above two equations and after simplification obtain $y^2 z^4 - (1+x^2y^2+y^4) z^2 + y^2 = 0$, which gives the lemma.
\end{proof}

Hence, we get 
$$\zeta(t)^2 = \frac{1+\zeta^4t^2+\zeta^4 + \sqrt{(1+\zeta^4t^2+\zeta^4)^2-4\zeta^4}}{2\zeta^2}, \qquad (1-\zeta^2(t) \zeta^2) \sin(\delta(t)) = \zeta^2t \cos(\delta(t)).$$
Using implicit differentiation, we get 
$$\zeta'(0) = 0, \qquad \delta'(0) = \frac{\zeta^2}{1-\zeta^4}.$$ 
Now, comparing the $(2,1)$ coefficient of both sides of \eqref{some-eqn} and using implicit differentiation, we get
$$\theta'(0) = -\frac{\delta'(0)}{\zeta^2} = - \frac 1{1-\zeta^4}.$$
Substituting into \eqref{X-action}, we get
\begin{equation}\label{E-action}
(E.B_0)(\mat{\zeta}{}{}{\zeta^{-1}}) = \frac{-i\ell}{1+\zeta^2} f(\zeta).
\end{equation}
\vskip 0.2in
\noindent{\bf $X = F$ case}:  Let $X=F$. Then $\exp(tF) = \mat{1}{}{t}{1}$. We have
$$\mat{\zeta}{}{}{\zeta^{-1}} \mat{1}{}{t}{1} = \mat{}{-1}{1}{} \mat{\zeta^{-1}}{-\zeta^{-1}t}{}{\zeta} \mat{}{1}{-1}{} = r(3\pi/2) \mat{\zeta^{-1}}{-\zeta^{-1}t}{}{\zeta} r(\pi/2).$$
Arguing as in the $X=E$ case, we get
$$\zeta'(0) = 0, \qquad \delta'(0) = \frac{\zeta^2}{1-\zeta^4}, \qquad \theta'(0) = -\frac{\zeta^4}{1-\zeta^4},$$
which gives us
\begin{equation}\label{F-action}
(F.B_0)(\mat{\zeta}{}{}{\zeta^{-1}}) = \frac{i\ell\zeta^2}{1+\zeta^2} f(\zeta).
\end{equation}
Using the definition \eqref{L-operator-defn} and the formulas \eqref{D-action}, \eqref{E-action} and \eqref{F-action}, we get
\begin{equation}\label{L-action}
(L.B_0)(\mat{\zeta}{}{}{\zeta^{-1}}) = \frac 12 \Big(\zeta f'(\zeta) - \ell \frac{1-\zeta^2}{1+\zeta^2} f(\zeta)\Big).
\end{equation}

\begin{proposition}\label{arch-non-split-values}
Let $\Omega$ be a character of $T(\R)$ given by $(\mu, m) \in \C \times \Z$ as defined in \eqref{Omega-defn-arch-ns}. For $\ell > 0$, let ${\mathcal D}_\mu(\ell)$ be the lowest weight discrete series module of $H(\R)$. If ${\mathcal D}_\mu(\ell)$ has a $(\Omega, S)$-Waldspurger model then the model is unique and $m = \pm \ell$. The lowest weight vector $B_0$ in the $(\Omega, S)$-Waldspurger model of ${\mathcal D}_\mu(\ell)$ is given by
\begin{equation}\label{arch-ns-values-eqn}
B_0(g) =  \gamma^\mu e^{i \ell(\delta + \theta)} \Big(\frac{\zeta}{1+\zeta^2}\Big)^\ell,
\end{equation}
if
$$ m=\ell, g=\mat{\gamma}{}{}{\gamma} r(\delta)\mat{\zeta}{}{}{\zeta^{-1}} r(\theta), \qquad \text{ OR } \qquad m=-\ell, g=\mat{\gamma}{}{}{\gamma} r(\delta)\mat{-\zeta}{}{}{\zeta^{-1}} r(\theta),$$
and $B_0(g) = 0$, otherwise.
\end{proposition}
\begin{proof}
Setting $(L.B_0) \equiv 0$ and \eqref{L-action} gives us the formula for $B_0$ above in the case $\ell = m$. The case $\ell = -m$ is very similar. Since ${\mathcal D}_\mu(\ell)$ is generated by the lowest weight vector $B_0$, we get the uniqueness of the Waldspurger model. 
\end{proof}


\subsection{The split case}
Let $S=\mat{-1}{}{}{1}$. Then
$$T(\R) = \{ \mat{x}{y}{y}{x} : x, y \in \R, x^2-y^2 \neq 0\}.$$
We have
$$T(\R) \ni t \mapsto t_0^{-1} \mat{x}{y}{y}{x} t_0 = \mat{x+y}{}{}{x-y} \simeq \R^\times \times \R^\times, \qquad \text{ where } t_0 = \mat{1}{1}{1}{-1}.$$
$$N = \{ \mat{1}{\zeta}{}{1} : \zeta \in \R\} \text{ and } A = \{ \mat{u}{}{}{v} : u, v \in \R^\times\}.$$
Using the Iwasawa decomposition, we get
\begin{equation}\label{GL2-split}
H(\R) = T(\R) t_0 N \SO(2).
\end{equation}
We have the character $\Omega$ of $T(\R)$ given by
\begin{equation}\label{arch-s-omega-defn}
\Omega(\mat{x}{y}{y}{x}) = \Omega(t_0 \mat{x+y}{}{}{x-y} t_0^{-1}) = \sgn(x+y)^{\epsilon_1}|x+y|^{\mu_1} \sgn(x-y)^{\epsilon_2}|x-y|^{\mu_2},
\end{equation}
with $\mu_1, \mu_2 \in \C, \epsilon_1, \epsilon_2 \in \{0,1\}$.

Let us assume that $\pi = {\mathcal D}_\mu(\ell)$ is given by its $(\Omega,S)$-Waldspurger model. Let $B_0 \in \pi$ be weight $\ell$ vector. Hence, we have
$$B_0(t g r(\theta)) = \Omega(t) e^{i \ell \theta} B_0(g).$$
Using the fact that $\mat{-1}{}{}{-1} \in \SO(2) \cap T(\R)$ and the central character of $\pi$, we get the necessary condition that 
\begin{equation}\label{Omega-split-cent-char-cond}
\epsilon_1 + \epsilon_2 \equiv \ell \pmod{2} \text{ and } \mu_1 + \mu_2 = \mu.
\end{equation}
Let us set
$$f(\zeta) := B_0(t_0 \mat{1}{\zeta}{}{1}).$$
For $X \in {\mathfrak g}$, we have
$$t_0 \mat{1}{\zeta}{}{1} \exp(tX) = \mat{x(t)}{y(t)}{y(t)}{x(t)} t_0 \mat{1}{\zeta(t)}{}{1} r(\theta(t)),$$
where $x(t), y(t), \zeta(t)$ and $\theta(t)$ are smooth functions such that $x(0) = 1, y(0) = \theta(0) = 0$ and $\zeta(0) = \zeta$ and $x(t)\pm y(t) >0$. Hence
\begin{align}
(X.B_0)(t_0 \mat{1}{\zeta}{}{1}) &= \frac d{dt} \Big|_{t=0} B_0(t_0 \mat{1}{\zeta}{}{1} \exp(tX)) \nonumber \\
&= \frac d{dt} \Big|_{t=0} B_0(\mat{x(t)}{y(t)}{y(t)}{x(t)} t_0 \mat{1}{\zeta(t)}{}{1} r(\theta(t))) \nonumber \\
&= \Big(\mu_1(x'(0)+y'(0)) + \mu_2(x'(0)-y'(0)) + i \ell \theta'(0)\Big) f(\zeta) + \zeta'(0) f'(\zeta) \label{split-X-action}
\end{align}
\vskip 0.1in
\noindent {\bf $X=D$ case}: Let $X=D$. Then
$$t_0 \mat{1}{\zeta}{}{1} \exp(tD) = t_0 \mat{1}{\zeta}{}{1} \mat{e^t}{}{}{e^{-t}} = t_0 \mat{e^t}{}{}{e^{-t}} t_0^{-1} t_0 \mat{1}{\zeta e^{-2t}}{}{1}.$$
Hence, $x(t)+y(t) = e^t, x(t)-y(t) = e^{-t}, \theta(t)=0$ and $\zeta(t)=\zeta e^{-2t}$. Applying \eqref{split-X-action}, we get
\begin{equation}\label{split-D-action}
(D.B_0)(t_0 \mat{1}{\zeta}{}{1}) = (\mu_1-\mu_2)f(\zeta) - 2\zeta f'(\zeta).
\end{equation}
\vskip 0.1in
\noindent
{\bf $X=E$ case}: Let $X=E$. Then
$$t_0 \mat{1}{\zeta}{}{1} \exp(tE) = t_0 \mat{1}{\zeta}{}{1} \mat{1}{t}{}{1} = t_0 \mat{1}{\zeta+t}{}{1}.$$
Hence, $x(t) \pm y(t) = 1, \theta(t) = 0$ and $\zeta(t) = \zeta + t$. Applying \eqref{split-X-action}, we get
\begin{equation}\label{split-E-action}
(E.B_0)(t_0 \mat{1}{\zeta}{}{1}) = f'(\zeta).
\end{equation}
\vskip 0.1in
\noindent
{\bf $X=F$ case}: Let $X=F$. Then
$$t_0 \mat{1}{\zeta}{}{1} \exp(tF) = t_0 \mat{1}{\zeta}{}{1} \mat{1}{}{t}{1} = t_0 \mat{1+\zeta t}{\zeta}{t}{1}.$$
Let $\mat{1+\zeta t}{\zeta}{t}{1} = \mat{a}{}{}{a^{-1}} \mat{1}{u}{}{1} k$, with $a \in \R_{>0}, u \in \R, k \in \SO(2)$. Applying both sides to $i$ as fractional linear transformation, and using that $\SO(2)$ stabilizes $i$, we get
$$\frac{(1+\zeta t)i+\zeta}{t i + 1} = a^2 i + a^2 u.$$
Hence, we get the system of equations
$$1+\zeta t = a^2 (1+tu), \qquad \zeta = a^2(u-t).$$
This give us
$$a=(1+t^2)^{-1/2}, \qquad u = (1+t^2) \zeta + t.$$
Hence, we have
$$t_0 \mat{1}{\zeta}{}{1} \exp(tF) = = t_0 \mat{1+\zeta t}{\zeta}{t}{1} = t_0 \mat{(1+t^2)^{-1/2}}{}{}{(1+t^2)^{1/2}} \mat{1}{(1+t^2) \zeta + t}{}{1} r(\theta(t)).$$
This gives us $x'(0) \pm y'(0) = 0, \theta'(0) = 0, \zeta'(0) = 1$. Applying \eqref{split-X-action}, we get
\begin{equation}\label{split-F-action}
(F.B_0)(t_0 \mat{1}{\zeta}{}{1}) = f'(\zeta).
\end{equation}
Using the definition \eqref{L-operator-defn} and formulas \eqref{split-D-action}, \eqref{split-E-action} and \eqref{split-F-action}, we get
\begin{equation}\label{split-L-action}
(L.B_0)(t_0 \mat{1}{\zeta}{}{1}) = \frac 12 \Big((\mu_1-\mu_2) f(\zeta) - (2\zeta + 2i) f'(\zeta)\Big).
\end{equation}

\begin{proposition}\label{arch-split-values}
Let $\Omega$ be a character of $T(\R)$  defined in \eqref{arch-s-omega-defn} for $\mu_1, \mu_2 \in \C$ and $\epsilon_1, \epsilon_2 \in \{0,1\}$. For $\ell > 0, \mu \in \C$, let ${\mathcal D}_\mu(\ell)$ be the lowest weight discrete series module of $H(\R)$. A necessary condition for  ${\mathcal D}_\mu(\ell)$ to  have a $(\Omega, S)$-Waldspurger model is $\epsilon_1+\epsilon_2 \equiv \ell \pmod{2}$ and $\mu_1+\mu_2 = \mu$. If a Waldspurger model exists, then it is unique. The lowest weight vector $B_0$ in the $(\Omega, S)$-Waldspurger model of ${\mathcal D}_\mu(\ell)$ is given by
\begin{equation}\label{arch-ns-values-eqn}
B_0(tt_0\mat{1}{\zeta}{}{1}r(\theta)) = \Omega(t) e^{i\ell \theta} (2i+2\zeta)^{\frac{\mu_1-\mu_2}2},
\end{equation}
for all $t \in T(\R), \zeta, \theta \in \R$. 
\end{proposition}

\begin{remark}
One can consider another matrix $S' = \alpha {}^tM S M$, with $\alpha \in \R^\times$ and $M \in \GL(2,\R)$, instead of $S = \mat{-1}{}{}{1}$. The torus $T_{S'} = \{g \in H(\R) : {}^tg S' g = \det(g) S'\}$ is given by $T_{S'} = M^{-1} T_S M$. The character $\Omega'$ of $T_{S'}$ corresponding to the character $\Omega$ of $T_S$ is given by $\Omega'(t') := \Omega(Mt'M^{-1})$. If $B$ is an element of a $(S, \Omega)$-Waldspurger model of $\pi$, then
\begin{equation}\label{change-model-eqn}
B'(g) := B(Mg), g \in H(\R),
\end{equation}
is an element of a $(S', \Omega')$-Waldspurger model of $\pi$. In the section below, we will make some special choices of $S'$ and will use Proposition \ref{arch-split-values} and (\ref{change-model-eqn}) to obtain the explicit formulas of the weight $\ell$ vectors in the $(S', \Omega')$-Waldspurger model of $\pi$.
\end{remark} 

\subsection{The local archimedean integral: the split case}
In this section, we will compute the local archimedean integral. We will use Proposition \ref{arch-split-values}, for values of the weight $\ell$ vector in a Waldspurger model of $\pi$. In Proposition \ref{arch-split-values}, we considered the torus to be the stabilizer of $S=\mat{-1}{}{}{1}$. For our global computation in the next section, we will need to consider a more general choice of $S$, which we will give now. Let $D > 0$ be a fundamental discriminant and set
\begin{equation}\label{S-D-defn}
S(D) := \begin{cases} \mat{\frac{-D}4}{}{}{1} & \text{ if } D \equiv 0 \pmod{4}; \\ \mat{\frac{1-D}4}{\frac 12}{\frac 12}{1} & \text{ if } D \equiv 1 \pmod{4}.\end{cases}
\end{equation}
First, assume that $D \equiv 0 \pmod{4}$. In this case, we have $S(D) = {}^tMSM$, where $M = \mat{\sqrt{D}/2}{}{}{1}$. Let $T_{S(D)}$ be the torus that is defined as the stabilizer of $S(D)$ in $\GL(2)$. Then $T_{S(D)} = M^{-1} T_S M$. Define $\Omega_D : T_{S(D)} \rightarrow \C$ by $\Omega_D(t) = \Omega(MtM^{-1})$. Let $B_D$ be the weight $\ell$ vector in a $(S(D), \Omega_D)$-Waldspurger model for $\pi$. Then we have
\begin{equation}\label{B-D-formula}
B_D(g) = B_0(Mg),
\end{equation}
where the values of $B_0(g)$ are given in Proposition \ref{arch-split-values}.

We want to compute the following integral
$$Z_\infty(s) = \int\limits_{T(\R) \backslash H(\R)} f(\eta h, s) \overline{B_D(h)} dh.$$
The measures are normalized as follows. For a function $\varphi$ on $H(\R)$, we have
$$\int\limits_{\GL(2,\R)^+} \varphi(g) dg = \int\limits_0^\infty \int\limits_\R \int\limits_0^\infty \int\limits_{\SO(2,\R)} \varphi(\mat{u}{}{}{u} \mat{v}{}{}{v^{-1}} \mat{1}{\zeta}{}{1} k)u^{-1} v^{-1} dk\,\,dv\,\,d\zeta\,\,du,$$
where $du, dv, d\zeta$ are the usual Lebesgue measures and $\int_{\SO(2,\R)} dk = 1$. Hence, for a function $\varphi$ which is left invariant under $T_{S(D)}(\R)$, using \eqref{GL2-split}, we have
$$\int\limits_{T(\R) \backslash H(\R)} \varphi(h) dh = \int\limits_\R \int\limits_{\SO(2,\R)} \varphi(M^{-1}t_0 \mat{1}{\zeta}{}{1} k) dk \,\,d\zeta,$$
where $t_0 = \mat{1}{1}{1}{-1}$ and $M = \mat{\sqrt{D}/2}{}{}{1}$.

Let us now make the following assumptions about the relevant representations. Let $\pi = {\mathcal D}(\ell)$, i.e., $\mu = 0$. Let $\ell_1, \ell_2$ be positive integers such that $\ell_1+\ell_2 = \ell$. Set 
\begin{equation}\label{Omega12-defn}
\Omega_1(x,y) = |x|^{\frac{\ell_1-1}2} |y|^{\frac{\ell_2-1}2} = \Omega_2(x,y)^{-1}, \qquad (x, y) \in (\R^\times)^2.
\end{equation}
A simple computation shows that $I(\Omega_1, \Omega_2, 0) = {\mathcal D}(\ell_1) \otimes {\mathcal D}(\ell_2)$. We have,
\begin{equation}\label{Omega-split-formula}
\Omega(x,y) = |x|^{\frac{\ell_1-\ell_2}2}  |y|^{\frac{\ell_2-\ell_1}2}.
\end{equation}
Let us choose a section $f \in I(\Omega_1, \Omega_2, s)$ which corresponds to a vector of weight $(\ell_1, \ell_2)$, i.e., we have 
\begin{equation}\label{f-choice-split}
f((\mat{u_1}{w_1}{}{z_1} r(\theta_1), \mat{u_2}{w_2}{}{z_2} r(\theta_2)),s) = \Omega_1(u_1,u_2) \Omega_2(z_1,z_2) \Big|\frac{u_1u_2}{z_1z_2}\Big|^{\frac 12+s} e^{i(\ell_1\theta_1+\ell_2\theta_2)}.
\end{equation}
The above formula, together with \eqref{B-D-formula}, gives us
\begin{align*}
Z_\infty(s) &= \int\limits_\R f(\eta M^{-1}t_0 \mat{1}{\zeta}{}{1}, s) \overline{B_D(M^{-1}t_0 \mat{1}{\zeta}{}{1})} d\zeta \\
&= \int\limits_\R f(\eta M^{-1}t_0 \mat{1}{\zeta}{}{1}, s) \overline{B_0(t_0 \mat{1}{\zeta}{}{1})} d\zeta
\end{align*}
In this case, $\eta = (\mat{1}{}{\sqrt{D}/2}{1}, \mat{1}{}{-\sqrt{D}/2}{1})$. We need to write the argument of $f$ above according to the Iwasawa decomposition. For this we have
\begin{align*}
\mat{1}{}{\sqrt{D}/2}{1} \mat{2/\sqrt{D}}{}{}{1} t_0 \mat{1}{\zeta}{}{1} &= \mat{2/\sqrt{D}}{}{}{1} \mat{-(1+\zeta^2)^{-\frac 12}}{\ast}{}{2(1+\zeta^2)^{\frac 12}} \mat{\frac \zeta{\sqrt{1+\zeta^2}}}{\frac{-1}{\sqrt{1+\zeta^2}}}{\frac{1}{\sqrt{1+\zeta^2}}}{\frac \zeta{\sqrt{1+\zeta^2}}},\\
\mat{1}{}{-\sqrt{D}/2}{1} \mat{2/\sqrt{D}}{}{}{1} t_0 \mat{1}{\zeta}{}{1} &= \mat{2/\sqrt{D}}{}{}{1}\mat{1}{\zeta+1}{}{-2}.
\end{align*}
Hence,
$$f(\eta M^{-1}  t_0 \mat{1}{\zeta}{}{1},s) =D^{-\frac{(\ell_1+\ell_2)}4-s} \frac{(\zeta + i)^{\ell_1}}{(1+\zeta^2)^{\ell_1+s}}.$$
We also have $\mu_1-\mu_2=\ell_1-\ell_2$. Hence, by Proposition \ref{arch-split-values}, we have
$$\overline{B(t_0 \mat{1}{\zeta}{}{1})} = 2^{\frac{\ell_1-\ell_2}2} (-i+\zeta)^{\frac{\ell_1-\ell_2}2} = 2^{\frac{\ell_1-\ell_2}2} \Big(\frac{1+\zeta^2}{i+\zeta}\Big)^{\frac{\ell_1-\ell_2}2}.$$
Hence, we get
\begin{equation}\label{Z-infty-eqn}
Z_\infty(s) = 2^{\frac{\ell_1-\ell_2}2} D^{-\frac{\ell}4-s} \int\limits_\R \frac{(i+\zeta)^{\frac{\ell}2}}{(1+\zeta^2)^{\frac{\ell}2+s}} d\zeta.
\end{equation} 
\begin{proposition}\label{arch-int-k1k2}
For positive integer $k$ and complex number $s$, set 
$$I(k, s) := \int\limits_{-\infty}^\infty \frac{(i+x)^{k}}{(1+x^2)^{k+s}} dx,$$
whenever the integral converges. Then, we have
\begin{equation}\label{arch-int-k1k2-eqn}
I(k, s) = \begin{cases} i \pi & \text{ if } k = 1, s = 0;\\
i^k2^{2-2s-k} \pi \frac{\Gamma(2s+k-1)}{\Gamma(s) \Gamma(k+s)} & \text{ if } {\rm Re}(2s+k) > 1. 
\end{cases}
\end{equation}
\end{proposition}
\begin{proof}
We have
$$I(1,0) = \int\limits_{-\infty}^\infty \frac{i+x}{1+x^2} dx = i \int\limits_{-\infty}^\infty \frac{1}{1+x^2} dx = i \arctan(x) |_{-\infty}^\infty = i \pi.$$
The general case is obtained by a suitable change of variable, a fairly complicated contour integral argument reducing the integral to the reciprocal of the beta function. 
\end{proof}

%
%

Let us remark that the special case of $k=1, s=0$ can also be obtained from the general formula above by taking the limit as $s$ approaches zero and the doubling formula for the gamma function. Substituting \eqref{arch-int-k1k2-eqn} into \eqref{Z-infty-eqn}, we get the following theorem.
\begin{theorem}\label{arch-split-thm}
Let $\pi = \mathcal{D}(\ell)$, where $\ell$ is a positive even integer. Let $\ell_1, \ell_2$ be two positive integers such that $\ell_1 + \ell_2 = \ell$. 
Let $\Omega_1, \Omega_2$ be characters of $\R^\times \times \R^\times$ given by \eqref{Omega12-defn}.  Let $\Omega$ be given by \eqref{Omega-split-formula}. For $D > 0$, a fundamental discriminant, let $S(D)$ be defined by \eqref{S-D-defn}. Let $\pi$ be given by its $(S(D), \Omega_{S(D)})$-Waldspurger model and let $B_D$ be a weight $\ell$ vector in $\pi$. Let $f \in I(\Omega_1, \Omega_2, s)$ be as defined in \eqref{f-choice-split}. Then,  we have
\begin{equation}\label{arch-split-formula}
Z_\infty(s) = \begin{cases} iD^{-1/2} \pi & \text{ if } \ell = 2, s= 0; \\ 
2^{2-2s-\ell_2} D^{-\frac{\ell}4-s} i^{\frac{\ell}2} \pi \frac{\Gamma(2s+\frac{\ell}2-1)}{\Gamma(s) \Gamma(\frac{\ell}2+s)} & \text{ if } {\rm Re}(2s+\frac{\ell}2) > 1.\end{cases}
 \end{equation}
 In particular, 
 $$Z_\infty(0) = 0 \text{ if } \ell > 2.$$
\end{theorem}
\begin{proof} The case $D \equiv 0 \pmod{4}$ follows from the computations above the statement of the theorem. The $D \equiv 1 \pmod{4}$ follows exactly as above noting that $\mat{\frac{1-D}4}{\frac 12}{\frac 12}{1} = \mat{1}{\frac 12}{}{1} \mat{\frac{-D}4}{}{}{1} \mat{1}{}{\frac 12}{1}$.
\end{proof}

\section{The global integral}\label{GlobalIntegral}
In this section, we will prove the main global theorem of the paper. We will specify the choices precisely and put together the local results from previous sections to obtain a formula for the global integral. We will also obtain a classical version of the integral formula rewriting the integral as the Petersson inner product of classical holomorphic modular forms.
\subsection{The main global theorem}\label{global-section}
Let us make the following assumptions. Let $F = \Q, L = \Q(\sqrt{D}),$ with $D > 0$ a fundamental discriminant. Let us set 
$$S(D) = \begin{cases} \mat{\frac{-D}4}{}{}{1} & \text{ if } D \equiv 0 \pmod{4};\\ & \\
\mat{\frac{1-D}4}{\frac 12}{\frac 12}{1} & \text{ if } D \equiv 1 \pmod{4}.\end{cases}$$
Let $\pi = \otimes' \pi_p$ be an irreducible cuspidal representation of $H(\A)$ and $N$ a square-free positive integer. 
\begin{itemize}
\item For $p \nmid N$, let $\pi_p$ be an unramified representation.

\item  For $p | N$, let $\pi_p$ be a twist of the Steinberg representation by an unramified character $\chi_p$. 

\item Let $\pi_\infty$ be the holomorphic discrete series representation ${\mathcal D}(\ell)$, with lowest weight $\ell$, a positive even integer. 
\end{itemize}
Let $\Omega_1, \Omega_2 : \A_L^\times \rightarrow \C^\times$ be two characters satisfying the following properties.
\begin{itemize}
\item Let $\Omega_1 \Omega_2 |_{\A^\times} = \omega_\pi$, where $\omega_\pi$ is the central character of $\pi$.

\item Let $N' | N$ be a positive integer. If $v \nmid N'$, then both $\Omega_{1,v}$ and $\Omega_{2,v}$ are unramified. If $v | N'$ then  assume that $c(\Omega_{1,v}) = 1, c(\Omega_{2,v}) = 0$. 

\item For $x, y \in \R^\times$, let $\Omega_{1,\infty}(x,y) = |x|^{\ell/2-1} |y|^{\ell/2-1} = \Omega_{2,\infty}^{-1}(x,y)$.

\end{itemize}
Note that, one can show that characters satisfying the above conditions do exist. Let $\Omega(z) = \Omega_1^{-1}(\bar{z}) \Omega_2^{-1}(z)$ for $z \in \A_L^\times$. Let us make the following assumptions.
\begin{itemize}
\item For every $p \leq \infty$, the local representation $\pi_p$ has a $(S(D), \bar\Omega_p)$-Waldspurger model. Note that the choices above imply that  this condition reduces to the following. If $p | (N/N')$, then either $p$ is split in $L$, or $p$ is ramified in $L$ and $\bar\Omega_p = \chi_p' \chi_p \circ N_{L_p/\Q_p}$, where $\chi_p'$ is the unique quadratic unramified character of $\Q_p^\times$.

\item Assume that $L(\frac 12, {\rm BC}(\pi) \times \bar\Omega) \neq 0$, where ${\rm BC}(\pi)$ is the base change of $\pi$ to $H(\A_L)$. 
\end{itemize} 
These two assumptions together imply that $\pi$ has a  non-zero global $(S(D), \bar\Omega)$-Waldspurger model. Let $\varphi  \in \pi$ and $B_\varphi = \otimes B_p$ be such that $B_p$ is in the $(S(D), \bar\Omega_p)$-Waldspurger model of $\pi_p$. Alternatively, $\bar{B_p}$ is in the $(S(D), \Omega_p)$-Waldspurger model of $\tilde\pi_p$, the contragredient representation of $\pi_p$. Choose $\varphi$  such that, for any $p < \infty$, we have $B_p$ is the newform in $\pi_p$, and $\varphi_\infty$ is the weight $\ell$ vector in $\pi_\infty$. These local functions are normalized as follows: 
\begin{itemize}
\item If $p \nmid N$ then $B_p(1) = 1$.

\item If $p | (N/N'), L_p = \Q_p \oplus \Q_p, \Omega_p(1,\varpi_p) = \bar\chi_p(\varpi_p)$, then $B_p(\mat{1}{}{u_1}{1}) = 1$. Here,
$u_1 = \sqrt{D}/2$ if $D \equiv 0 \pmod{4}$ and $u_1 = (1+\sqrt{D})/2$ if $D \equiv 1 \pmod{4}$.

\item If $p < \infty$ and does not satisfy any of the conditions above, then $B_p(\mat{}{1}{-1}{}) = 1$.

\item For $p = \infty$, we have $B_\infty(M_D^{-1}t_0) = 1$, where $t_0 = \mat{1}{1}{1}{-1}$ and $M_D = \mat{\sqrt{D}/2}{}{}{1}$ if $D \equiv 0 \pmod{4}$ and $M_D = \mat{\sqrt{D}/2}{}{1/2}{1}$ if $D \equiv 1 \pmod{4}$.
\end{itemize}

 Let us choose the section $f( \cdot, s) = \otimes f_v( \cdot, s) \in I(\Omega_1, \Omega_2, s)$ as follows. We will write $f_p$ for $\otimes_{v|p} f_v$. If $p \nmid N$, then $f_p$ is the spherical vector in the local representation normalized by $f_p(1) = 1$. If $p | N'$ then $f_p$ is the newform in the local representation normalized so that $f_p(\mat{1}{}{1}{1}) = 1$. If $p | (N/N')$, then we choose $f_p$ to be the translate of the spherical vector, normalized to be $1$ at the identity, by $\mat{\varpi_{L_v}^{-1}}{}{}{1}$. For $p = \infty$, we choose $f_\infty$ to be the vector of weight $(\ell/2, \ell/2)$ given by (\ref{f-choice-split}) with $\ell_1 = \ell_2 = \ell/2$. The next theorem computes the following global integral
$$Z(s, f, \bar\varphi) = \int\limits_{Z_H(\A) H(\Q) \backslash H(\A)} E(h, s, f) \bar\varphi(h) dh.$$

\begin{theorem}\label{global-int-thm}
Let the notations and choices of local vectors be as above. Then, we have
$$Z(s, f, \bar\varphi) = \frac{L(2s+\frac 12, \tilde\pi \times \Omega_1|_{\A^\times})}{L(2s+1, \Omega_1 \Omega_2^{-1})} \prod\limits_{p \leq \infty} Y_p(s),$$
where, for $p < \infty$, we have
$$Y_p(s) = \begin{cases} 1 & \text{ if } p \nmid N; \\ 
\frac{p-\Big(\frac L{\p}\Big)}{p+1} L_p(2s+1, \Omega_1 \Omega_2^{-1}) & \text{ if } p | N'; \\
\frac{-\Omega_1(\varpi_{L_p})^{-1} p^{s+\frac 32}}{p+1} & \text{ if } p | \frac N{N'}, \Big(\frac L{\p}\Big) = 0;\\
\frac{-\Omega_1(\varpi_p,1)^{-1} p^{s+\frac 12}}{p+1} & \text{ if } p | \frac N{N'}, \Big(\frac L{\p}\Big) = 1, \Omega(1,\varpi_p) = \bar\chi_p(\varpi_p);\\
\frac{-(p-1)\Omega_1(\varpi_p,1)^{-1} p^{s+\frac 12}}{(p+1)(1-\bar\chi_p(\varpi_p)^{-1}\Omega(1,\varpi_p))} & \text{ if } p | \frac N{N'}, \Big(\frac L{\p}\Big) = 1, \Omega(1,\varpi_p) \neq \bar\chi_p(\varpi_p), \end{cases}
$$
and
$$Y_\infty(s) = \begin{cases} iD^{-1/2} \pi & \text{ if } \ell = 2, s= 0; \\ 
2^{2-2s-\ell_2} D^{-\frac{\ell}4-s} i^{\frac{\ell}2} \pi \frac{\Gamma(2s+\frac{\ell}2-1)}{\Gamma(s) \Gamma(\frac{\ell}2+s)} & \text{ if } {\rm Re}(2s+\frac{\ell}2) > 1.\end{cases}
$$
Here, $\tilde\pi$ is the contragredient representation of $\pi$.
\end{theorem}
\begin{proof}
The theorem follows from Theorems \ref{local-unram-int-thm}, \ref{stein-ram-int-thm}, \ref{stein-unramif-int-thm} and \ref{arch-split-thm}. 
\end{proof}

\subsection{Petersson norm of classical modular forms}
In this section, we will realize the global integral $Z(s, f, \bar{\varphi})$ as the Petersson inner product of classical modular forms on the complex upper half plane ${\mathcal H} := \{ x+iy \in \C : y > 0\}$. Let $(\tau_1, \tau_2) \in {\mathcal H}^2$ and let $g_1, g_2 \in \SL(2,\R)$ such that $g_j \langle i \rangle = \tau_j$. Here, we have $g\langle \tau \rangle = (a \tau + b)/(c \tau + d)$ for $\tau \in {\mathcal H}$ and $g = \mat{a}{b}{c}{d} \in H(\R)$. Set $g = \otimes_v g_v \in H(\A_L)$ by $g_v = 1$ for $v \nmid \infty$ and $g_\infty = (g_1, g_2)$. Define the Eisenstein series ${\mathcal E}((\tau_1, \tau_2), s, f) : {\mathcal H}^2 \rightarrow \C$ by
\begin{equation}\label{classical-Hilb-Eis-ser-defn}
{\mathcal E}((\tau_1, \tau_2), s, f) := J(g_1, i)^{\ell/2} J(g_2, i)^{\ell/2} E(g, s, f),
\end{equation}
where $J(\mat{a}{b}{c}{d}, \tau) := c \tau + d$. Note that the right hand side above is well-defined by the choice of the section $f$. In fact, if $\tau_j = x_j + i y_j$, then we can choose $g_j = \mat{1}{x_j}{}{1} \mat{\sqrt{y_j}}{}{}{1/\sqrt{y_j}}$. In this case, $J(g_j, i) = y_j^{-1/2}$. Let $\Phi$ be the cusp form on $\mathcal H$ corresponding to $\varphi$ from the previous section. 

For two smooth functions $f_1, f_2$ on $\mathcal H$ of weight $\ell$ with respect to $\Gamma_0(N)$, at least one of which is rapidly decreasing at $\infty$, we define the Petersson inner product by
\begin{equation}\label{Petersson-inner-product}
\langle f_1, f_2 \rangle := \frac 1{{\rm vol}(\Gamma_0(N) \backslash {\mathcal H})} \int\limits_{\Gamma_0(N) \backslash {\mathcal H}} f_1(\tau) \overline{f_2(\tau)}y^\ell \frac{dx dy}{y^2}.
\end{equation}

\begin{proposition}\label{Petersson-norm-prop}
With notations as in \ref{global-int-thm}, we have
\begin{equation}
Z(s, f, \bar\varphi) ={\rm vol}(\Gamma_0(N) \backslash {\mathcal H})  \langle {\mathcal E}|_{\Delta \mathcal H},  \Phi \rangle.
\end{equation}
\end{proposition}
\begin{proof}
The proposition follows from 
$$Z_H(\A) H(\Q) \backslash H(\A) / \SO(2,\R) K_0(N) \simeq Z_H(\R) \Gamma_0(N) \backslash H(\R)^+/\SO(2,\R) \simeq \Gamma_0(N) \backslash {\mathcal H},$$
and, for $h \in \SL(2, \R)$ and $h \langle i \rangle = \tau$, we have
$$E(h,s,f) \bar\varphi(h) = J(h,i)^{-\ell} {\mathcal E}((\tau, \tau), s, f) J(h, i)^{-\ell} \overline{\Phi(\tau)} = {\mathcal E}((\tau, \tau), s, f) \overline{\Phi(\tau)} y^{\ell}.$$
Here, $K_0(N)$ is defined in (\ref{K0-defn}) below.
\end{proof}

\subsection{Special cases arising from Tonghai Yang's paper}
In \cite{Yang}, Tonghai Yang has considered Hilbert Eisenstein series obtained from certain specific choices of the characters $\Omega_1$ and $\Omega_2$. Let us explain the setup of Theorem 1.2 of \cite{Yang}. Let us first remark that, in \cite{Yang}, an Eisenstein series is constructed on $\SL(2)$, whereas, in this paper, we are constructing Eisenstein series on $\GL(2)$. 

Let $L=\Q(\sqrt{D})$ be a real quadratic extension of $\Q$ and let $K$ be an imaginary quadratic extension of $L$. Let $\chi_{K/L}$ be the character of $L$ associated to the extension $K/L$. Let $\mathcal N$ be a square-free integral ideal of $L$ such that all its prime factors are inert in $K$. Let $\Omega_1 = \chi_{K/L}$ and $\Omega_2 = 1$. Let $N$ be a positive square-free integer such that $d_{K/L} \mathcal N \cap \Z = N \Z$ and $N'$ be an integer such that $d_{K/L} \cap \Z = N'\Z$. Here, $d_{K/L}$ is the discriminant of $K/L$. Let $\psi$ be the Hecke character corresponing to $\chi_{K/L}$. Let $E(g, s, f)$ be the Eisenstein series on $H(\A_L)$, with the section $f(\ast, s) \in I(\chi_{K/L}, 1, s)$ as in Section \ref{global-section}. Let ${\mathcal E}((\tau_1, \tau_2), s, f) := J(g_1, i) J(g_2, i) E(g, s, f)$ be the Eisenstein series on $\mathcal H^2$ as defined in (\ref{classical-Hilb-Eis-ser-defn}). Theorem 1.2 part 2) of \cite{Yang} states that, as a function of $(\tau_1, \tau_2)$, the Eisenstein series ${\mathcal E}((\tau_1, \tau_2), s, f)$ is a Hilbert modular form (non-holomorphic) of weight $(1,1)$, level $d_{K/L} \mathcal N$ and character $\psi$. Furthermore, part 3) of Theorem 1.2 in \cite{Yang} states that, when non-zero, the Eisenstein series ${\mathcal E}((\tau_1, \tau_2), 0, f)$ is holomorphic. 

Let $\Phi \in S_2(\Gamma_0(N), \psi)$ be a cusp form of weight $2$, level $N$ and nebentypus character $\psi$. Here, we have used the same notation for the Dirichlet character obtained by restriction  of $\psi$. Let $\omega$ be the character of $\Q^\times \backslash \A^\times$ corresponding to $\psi$. Note that $\omega = \chi_{K/L} |_{\A_\Q^\times}$. Let 
\begin{equation}\label{K0-defn}
K_0(N) := \prod\limits_{p < \infty} K_p(N), \text{ where } K_p(N) = \begin{cases} H(\Z_p) & \text{ if } p \nmid N;\\ H(\Z_p) \cap \mat{\Z_p}{\Z_p}{p\Z_p}{\Z_p} & \text{ if } p | N. \end{cases}
\end{equation}
Define the character $\lambda$ of $K_0(N)$ by
$$\lambda(\mat{a}{b}{c}{d}) := \prod\limits_{p | N} \omega_p^{-1}(d_p).$$
The function $\varphi : H(\A) \rightarrow \C$ corresponding to $\Phi$ is given by the formula
\begin{equation}\label{classical-adelic-defn}
\varphi(g) = \varphi(\gamma g_\infty k_0) :=  \lambda(k_0) \frac{\det(g_\infty)}{J(g_\infty, i)^2} \Phi(g_\infty \langle i \rangle).
\end{equation}
Here, using strong approximation, we have written $g = \gamma g_\infty k_0$, with $\gamma \in H(\Q), g_\infty \in \GL(2, \R)^+$ and $k_0 \in K_0(N)$. Assume that $\Phi$ is a Hecke eigenform. Let $\pi$ be the irreducible cuspidal automorphic representation of $H(\A_{\Q})$ generated by the right translates of $\varphi$. The central character of $\pi$ is given precisely by $\omega$. Assume that, for every $p \leq \infty$, the local representation $\pi_p$ has a $(S(D), (\chi_{K/L})_p)$-Waldspurger model. In Section 5 of \cite{Yang}, several special choices of $K$ and $L$ are made which automatically guarantee this local condition. Also assume that $L(\frac 12, {\rm BC}(\pi) \times \chi_{K/L}) \neq 0$. Theorem \ref{global-int-thm} and Proposition \ref{Petersson-norm-prop} gives us the following theorem.


\begin{theorem}\label{Classical-thm}
Let the notations be as above. Then we have
$$\langle {\mathcal E}|_{\Delta \mathcal H},  \Phi \rangle = i \pi D^{-\frac 12}  {\rm vol}(\Gamma_0(N) \backslash {\mathcal H}) \frac{L(1/2, \pi)}{L(1, \chi_{K/L})} \prod\limits_{p < \infty} Y_p(0),$$
where $Y_p(s)$ is the same as in the statement of Theorem \ref{global-int-thm}.
\end{theorem}
We get the following corollary on non-vanishing of the Petersson inner product.
\begin{corollary}\label{classical-cor}
Let the notations be as above. Then, we have $\langle {\mathcal E}|_{\Delta \mathcal H},  \Phi \rangle \neq 0$ if and only if $L(1/2, \pi) \neq 0$ and $L(1/2, {\rm BC}(\pi) \times \chi_{K/L}) \neq 0$.
\end{corollary}
Note that, by results of Friedberg and Hoffstein in \cite{FH}, given a $\pi$, one can obtain characters $\chi_{K/L}$ such that $L(1/2, {\rm BC}(\pi) \times \chi_{K/L}) \neq 0$.

%
%
%
%
%
%
%
%
\section*{Acknowledgements}
First and foremost, the authors would like to thank Yingkun Li for providing the original motivation for this project, as well as many helpful comments along the way. The authors would also like to thank Alok Shukla for helpful advice in evaluating the integral in Prop. \ref{arch-int-k1k2}.

\end{document}